\documentclass[11pt,reqno,tbtags]{amsart}
\usepackage{subcaption}
\usepackage{graphicx}
\usepackage{amsthm,amssymb,amsfonts,amsmath}
\usepackage{amscd} 
\usepackage{mathrsfs}
\usepackage[numbers, sort&compress]{natbib} 
\usepackage{vmargin}
\usepackage{setspace}
\usepackage{paralist}
\usepackage[pagebackref=true]{hyperref}
\usepackage[usenames,dvipsnames]{color}
\usepackage{tcolorbox}
\usepackage[normalem]{ulem}
\usepackage{stmaryrd} 
\usepackage[refpage,noprefix,intoc]{nomencl} 
\usepackage{bbm}
\providecommand{\R}{}
\providecommand{\Z}{}
\providecommand{\N}{}

\renewcommand{\R}{\mathbb{R}}
\renewcommand{\Z}{\mathbb{Z}}
\renewcommand{\N}{{\mathbb N}}

\newcommand{\pr}{\mathbb{P}}

\newcommand{\E}{\mathbb E}

\newcommand\T{{\mathcal T}}

\providecommand{\ora}[1]{}
\renewcommand{\ora}[1]{\overrightarrow{#1}}
\DeclareRobustCommand{\SkipTocEntry}[5]{} 

\reversemarginpar
\definecolor{cser}{RGB}{0,145,135}

\definecolor{cjer}{RGB}{203,96,21}

\newtheorem{thm}{Theorem}
\newtheorem{lem}[thm]{Lemma}
\newtheorem{prop}[thm]{Proposition}
\newtheorem{cor}[thm]{Corollary}

\newtheorem{definition}[thm]{Definition}

\theoremstyle{definition}
\newtheorem{rem}[thm]{Remark}
\numberwithin{equation}{section}
\numberwithin{thm}{section}

\graphicspath{ {./images/} }
\usepackage{wrapfig}

\newcommand{\floor}[1]{\lfloor #1 \rfloor}

\begin{document}
	
\title{Convergence and compactness of discrete aggregation trees} 
\author[S. Donderwinkel]{Serte Donderwinkel}
\address{Bernoulli Institute for Bernoulli Institute for Mathematics, Computer Science and Artificial Intelligence and CogniGron, Groningen, The Netherlands}
\author[J. van Haastert]{Jeroen van Haastert}
\address{Faculty of Mathematics, Ruhr University Bochum, 44780 Bochum, Germany}
\email{s.a.donderwinkel@rug.nl}
\email{jeroen.vanhaastert@ruhr-uni-bochum.de}

\date{September 11, 2026} 


\begin{abstract} 
We study a class of random aggregation trees that generalizes the discrete stick-breaking construction of the uniform labelled tree (Aldous, 1991). To sample the tree of size $n$, vertices are added sequentially, with the $i$th vertex starting a new branch with a prescribed probability $f(n,i)$; otherwise, it extends the current branch. We also define a continuum analogue by replacing the Poisson point process of intensity $tdt$ in the construction of the Brownian continuum random tree by one of intensity $f(t)dt$.

We establish scaling limits for two families of discrete aggregation trees in the Gromov--Hausdorff--Prokhorov topology. When $f(n,i)=(i/n)^\beta$, with $\beta>0$, after rescaling the graph distance by $n^{-\beta/(\beta+1)}$, the random tree converges to the compact aggregation tree with $f(t)=t^\beta$. This recovers convergence to the  Brownian continuum random tree when $\beta=1$ (Aldous, 1991), as well as  scaling limits of choice spanning trees for integer $\beta$ (Archer and Shalev, 2024). We also prove convergence under rescaling to the compact aggregation tree with $f(t)=\log^\gamma(1+t)$ for every $\gamma>1$.

Finally, we identify necessary conditions for compactness. Consequently, the threshold $\gamma>1$ in the logarithmic family is sharp. These results provide insight into an open problem on compactness criteria for random aggregation trees (Curien and Haas, 2014).
\end{abstract}

\maketitle
	\section{Introduction}\label{Section:introduction}
    It is well known that the scaling limit of uniformly chosen labelled rooted trees is the continuum random tree (CRT). This was first shown by Aldous in \cite{AldousCRT1} via the stickbreaking construction: a uniformly distributed labelled tree was constructed by starting with a root vertex labelled $1$ and iteratively attaching vertices with labels $\{2,\dots,n\}$ where vertex $i+1$ is attached to the vertex with label $\min(U_i, i)$ where $U_i \in_u [n]$. This constructs the tree branch by branch, where vertex $i+1$ is added to the current branch with probability $1 - i/n$ and starts a new branch at a uniform point of the already constructed tree with probability $i/n$. A uniformly labelled tree is obtained after randomly relabelling the vertices. 

    The branch-by-branch construction of the uniform tree on $n$ vertices allows for determining their scaling limit, which can be described as follows: Let $\eta$ be a Poisson point process (PPP) of intensity $t dt$ on the positive real line. The points of the PPP split up $\R_{\ge 0}$ into sticks, which we order according to their location on $\R_{\ge 0}$. We then iteratively construct a real tree by glueing the $i$th stick at a uniformly chosen point on the tree constructed from the previous $i-1$ sticks. The completion of the resulting object is a compact metric space that can be equipped with a natural probability measure. In \cite{AldousCRT1}, Aldous proved that the rescaled uniform labelled tree converges to the Brownian CRT; in modern terminology, this convergence may be formulated in the Gromov--Hausdorff--Prokhorov topology.

    The iterative construction, where new branches of the tree attach at a uniform position on the previously built tree, makes the uniform labelled tree an example of a \emph{discrete aggregation tree}, and, its scaling limit, the CRT, is a \emph{continuous aggregation tree}. See \cite{curien2016randomtreesconstructedaggregation} for the first use of the term `aggregation trees' in the literature. In this work, we examine a wider class of discrete aggregation trees, by varying the intensity of new branches forming.

    To be precise, we consider the following natural generalization of the discrete object: the $(i+1)$st vertex added to the tree starts a new branch from a uniformly random existing vertex with probability $f(n,i)$ for some function $f(n,i) : \{1,\dots,n-1\} \to [0,1]$. We call the resulting tree the \emph{$f(n,i)$-aggregation tree}. For the continuous object, we change the intensity measure in the point process used to construct the CRT to $f(t) dt$ for $f:[0, \infty) \to [0, \infty)$ a locally integrable function to obtain the $f(t)$-aggregation tree. Informally, for $0=C_0<C_1<\dots$ the points of a Poisson point process on $[0,\infty)$ with intensity $f(t)dt$, we let the $i$th `stick' have length $c_i=C_i-C_{i-1}$ and, for $i\ge 2$, we recursively glue an end of the $i$th stick on a uniform position of the tree formed by the previous sticks and thereafter take the completion of the resulting metric space. (In Theorems  \ref{thm:not_compact} and~\ref{thm:not_compact2} we show that the \emph{$f(t)$-aggregation tree} is not always compact.)
    
    Our first result is as follows.
    \begin{thm}
    \label{Theorem: main result 1 (polynomial case)}
        Let $f(n,i) = (i/n)^\beta$ for $\beta > 0$ and let $\T_n^f$ the $f(n,i)$-aggregation tree. Let $\T_\beta$ be the $t^\beta$-aggregation tree. Then, $\T_\beta$ is compact and
        $$\big(\T_n^f, n^{-\frac{\beta}{\beta+1}} d_n, \nu_n\big) \xrightarrow[n \to \infty]{d} \big(\T_\beta, d, \mu\big)$$
        in the Gromov--Hausdorff--Prokhorov topology, for $\nu_n$ the uniform vertex measure and $\mu$ a probability measure.  
    \end{thm}
    We formally define $\T_\beta$ and its measure $\mu$ in Section~\ref{Section: the polynomial case}. For $\beta = 1$, we recover Aldous' scaling limit of uniform labelled trees from \cite{AldousCRT1} and for $\beta \in \Z_{\ge 1}$ we recover the scaling limit of the uniform $\beta$-choice spanning trees from \cite{RandomChoiceTree}. See  \cite{RossWen2018} for a formal definition and the convergence result. 

    Our second result shows that we still get a compact scaling limit when the intensity of starting new branches increases much slower than polynomial in $i$. Here, the scale $n^{1/2}$ is chosen to ease notation and could be changed to, for example, $n^\beta$ without changing the proof.
    
    \begin{thm}\label{thm:log}
        Fix $\gamma > 1$. Let $g(n,i) = \ln^\gamma(in^{-\frac{1}{2}} + 1)n^{-\frac{1}{2}}$ and let $\T_n^g$ be the $g(n,i)$-aggregation tree. Let  $\T_\gamma$ be the  $\ln^\gamma(t+1)$-aggregation tree. Then $\T_\gamma$ is compact and
        $$\big(\T_n^g, n^{-\frac{1}{2}} d_n, \nu_n\big) \xrightarrow[n \to \infty]{d} \big(\T_\gamma, d, \mu\big),$$
        in the Gromov--Hausdorff--Prokhorov topology, for $\nu_n$ the uniform vertex measure and $\mu$ a probability measure.  
    \end{thm}
 We formally define $\T_\gamma$ and its measure $\mu$ in Section~\ref{Section: the logarithmic case}. This theorem partially resolves an open problem by Curien and Haas, by showing compactness of continuous aggregation trees in which the length of the longest stick with index exceeding $i$ decreases slower than any polynomial. 

We show that Theorem~\ref{thm:log} cannot be improved to smaller values of $\gamma$, by showing the existence of a barrier to compactness.
\begin{thm}\label{thm:not_compact}
   Let $f(t)$ be increasing and locally integrable such that $\int_1^\infty e^{-f(t)}dt = \infty$. Then, the $f(t)$-aggregation tree is almost surely non-compact.
\end{thm}
This has the following consequence, illustrating that the lower bound for $\gamma$ in Theorem~\ref{thm:log} is tight.
\begin{cor}
    For $\gamma\le 1$, the $\ln^\gamma(t+1)$-aggregation tree is not compact.
\end{cor}

Under the conditions of Theorem~\ref{thm:not_compact}, compactness of the aggregation tree is violated because the stick-lengths do not go to $0$ almost surely. Our next theorem shows that this is not a necessary condition, by providing an even weaker condition for non-compactness, in the case of $f(t)$ with appropriate regularity properties. 

\begin{thm}\label{thm:not_compact2}
Let $f(t)$ be increasing and locally integrable. Further, assume that $f(t)$ is slowly varying.  Then, the $f(t)$-aggregation tree is almost surely non-compact if 
\[\sum_{i=1}^\infty\frac{1}{if(i)}= \infty\]
\end{thm}

\begin{proof}
Let $\eta$ be a PPP of intensity $f(t)dt$ and order its points $C_1 < C_2 < \dots$ Let $c_i = C_i - C_{i-1}$ be the gap sizes between consecutive points in $\eta$. From \cite[Proposition 2.2]{curien2016randomtreesconstructedaggregation} we obtain a necessary condition for the $f(t)$-aggregation tree to be compact: $\sum_{i=1}^\infty c_i^2/C_i < \infty$. 
    Write $\Lambda(t) = \int_0^t f(s)ds$, so that $\Lambda(C_i)$ is a homogeneous PPP of intensity $1$. The strong law for the PPP gives $\Lambda(C_i) = i(1 + o(1))$ a.s. Using that $f$ is slowly varying, we obtain $\Lambda(t)\sim tf(t)$ as $t\to \infty$, so 
    $$C_i = (1 + o(1))\frac{i}{f(C_i)} \qquad\text{and}\qquad c_i=(1+o(1))\frac{E_i}{f(C_i)} \quad\text{a.s.},$$
    where $E_i = \Lambda(C_i) - \Lambda(C_{i-1})$ are i.i.d.\ $\operatorname{Exp}(1)$ random variables. In particular, $C_i\le i $ for all $i$ large enough, so $f(C_i)\le (1+o(1)) f(i)$.  Hence,
    $$\sum_{i=1}^n \frac{c_i^2}{C_i} \ge (1+o(1))\sum_{i=1}^n \frac{E_i^2}{if(i)} \xrightarrow[n \to \infty]{} \infty \quad\text{a.s.}$$
    and thus the $f(t)$-aggregation tree is not compact. 
    
\end{proof}
We now illustrate that this can indeed provide non-compact aggregation trees whose stick-lengths go to $0$. Take $f(t)=\ln(t+1)\ln(\ln(t+e))$. Since $f$ is increasing and $f(t)/\ln(t)\to \infty$ as $t\to \infty$, \cite[Equation 7]{asmussen2016timeinhomogeneitylongestgap} implies that the stick-lengths go to $0$, and it is easy to check that the conditions of Theorem~\ref{thm:not_compact2} are satisfied. Our Theorem~\ref{thm:log} shows that replacing the double logarithm in the definition of $f(t)$ by any positive power of a logarithm does guarantee compactness, illustrating the delicacy of our result.
    
    Our positive and negative results on compactness are related to Open Question 1 in \cite{curien2016randomtreesconstructedaggregation}, which asks for a necessary and sufficient compactness criterion in the quenched setting, when stick lengths are deterministic. A necessary and sufficient criterion is already known in the setting where the sticks have deterministic and decreasing length as shown in \cite{AminiDevroyeGriffithsOlver2017}. In this setting, the tree constructed from gluing stick $i$ with length $\ell (i)$ to the prior constructed tree at a uniform point is compact if and only if $\ell(i) = o(1/\ln(i))$.

    \subsection{Related work}

The stick-breaking construction originates in Aldous' work on the Brownian continuum random tree. Aldous used an inhomogeneous Poisson process with intensity $tdt$ to construct the Brownian CRT and proved that it is the scaling limit of uniform labelled trees \cite{AldousCRT1,CRTIII}. Aldous also showed that the CRT can be encoded by a normalized Brownian excursion, see \cite{CRTII}. 

Our continuum model belongs to the family of random trees constructed by aggregation introduced and studied by Curien and Haas \cite{curien2016randomtreesconstructedaggregation}. Given a deterministic sequence of positive stick lengths $(a_i)_{i\geq 1}$, their model is obtained by successively attaching the $i$th stick at a point sampled uniformly from the length measure on the previously constructed tree. They show that the tree is compact almost surely when either $a_i<i^{-\epsilon+o(1)}$ for some $\epsilon>0$ or the total stick length is finite. (The first condition shows that the $t^\beta$-aggregation tree is compact almost surely, but almost surely neither condition applies to the $\ln^\gamma(t+1)$-aggregation tree that we study.) They also obtain the Hausdorff dimension of the resulting compact tree under stronger conditions. Amini, Devroye, Griffiths and Olver obtained a necessary and sufficient compactness criterion when the stick lengths are deterministic and decreasing \cite{AminiDevroyeGriffithsOlver2017}. Haas subsequently obtained precise asymptotics for the height and for finite reduced subtrees in several noncompact regimes \cite{Haas2017}.  In contrast to \cite{curien2016randomtreesconstructedaggregation,AminiDevroyeGriffithsOlver2017,Haas2017}, the stick lengths in our continuum model are random. Finally, in \cite{Senizergues2019}, S\'enizergues generalizes continuous aggregation trees by gluing metric spaces rather than line-segments and studies their fractal properties.

Several works have obtained continuum aggregation trees as scaling limits of discrete random-tree models. Most closely related to the present paper is the work of Archer and Shalev \cite{RandomChoiceTree}, who, for integer $\beta$, introduce $\beta$-choice spanning trees through variants of the Aldous--Broder and Wilson algorithms that use choice random walks. Their discrete trees admit a branch-by-branch description closely related to our aggregation construction: as the tree is explored, a new branch is either extended or attached to the previously constructed tree, with the choice mechanism modifying the rate at which new branches are formed. They prove convergence to the (deterministically rescaled) $t^\beta$-aggregation tree (with $\beta$ integer). Our polynomial result therefore recovers their continuum limits, while placing them within a broader family of discrete aggregation trees.

The $t^\beta$-aggregation tree (with $\beta$ integer) also occurs, rather mysteriously, as the limit under rescaling of a family of discrete models studied by Ross and Wen \cite{RossWen2018}. They introduce inhomogeneous variants of R\'emy's algorithm and prove almost-sure convergence in the Gromov--Hausdorff--Prokhorov topology to continuous aggregation trees that, again, after deterministic rescaling, agree with the $t^\beta$-aggregation tree for integer $\beta$. Thus, the same continuum objects arise from their recursive growth procedure, from the choice spanning trees of Archer and Shalev, and from the aggregation trees studied here. 

Stick-breaking constructions have also been developed for other important families of continuum random trees. Goldschmidt and Haas constructed stable L\'evy trees by successively attaching line segments whose lengths are governed by a Mittag--Leffler Markov chain \cite{GoldschmidtHaas2015}. Inhomogeneous continuum random trees were introduced as scaling limits of $p$-trees by Camarri and Pitman \cite{CamarriPitman2000} and admit stick-breaking and exploration-process descriptions \cite{AldousPitmanICRT,AldousMiermontPitman}. Their compactness and fractal dimensions were studied by Blanc-Renaudie using a stick-breaking representation \cite{ICRTsecondpaper}. Furthermore, Blanc-Renaudie used stick-breaking constructions to study scaling limits of trees with a given degree sequence \cite{BlancRenaudie2021}. Addario-Berry and the first author used stick-breaking to obtain height bounds on trees with a given degree sequence, simply generated trees and branching processes \cite{height}. More recently, Goldschmidt and Hill introduced a stick-breaking construction to reprove Gromov--Hausdorff--Prokhorov convergence of size-conditioned branching processes to the $\alpha$-stable tree \cite{ICRTpreprint}. These constructions differ from ours both in the mechanism selecting attachment points and in the structure of the process governing the branch lengths.

Our results complement this literature in two directions. First, we derive scaling limits for polynomial intensities with arbitrary positive exponent and slowly varying logarithmic intensities. Second, we show compactness of a new family of aggregation trees, and find an obstruction to compactness formulated directly in terms of the continuum intensity.

	\section*{Acknowledgements}
	The authors would like to thank Gilles Bonnet, who served as the second supervisor for the master project on which this work is based, for his careful reading of the thesis and his useful comments. SD acknowledges
    the financial support of the CogniGron research center
    and the Ubbo Emmius Funds (University of Groningen). Her research was also supported by the Marie Sk\l{}odowska-Curie grant GraPhTra (Universality in phase transitions in random graphs), grant agreement ID 101211705.

    \section*{AI statement}
    Large language models were exclusively used for composing the literature review. 
\section{Preliminaries and Proof Outline}
\label{Section: preliminaries and proof outline}
\subsection{The discrete aggregation tree}
Let $f(n,\cdot) : \{1, \dots, n-1\} \to [0,1]$ be some function. For $i \in [n-1]$ let $R_i \sim \text{Ber}(f(n,i))$ be independent Bernoulli random variables with success probability $f(n,i)$. The $f(n,i)$-aggregation tree, denoted $\T_n^f$, is iteratively constructed on the vertices with labels in $[n]$ . Start with the root vertex labelled $1$. For $i \in [n-1]$ we have the following attachment rule
$$\begin{cases}
    \text{attach vertex } i+1 \text{ to vertex } i & \text{ if $R_i = 0$},\\
    \text{attach vertex } i+1 \text{ to a uniform vertex in } [i] & \text{ if $R_i = 1$}.
\end{cases}$$
Let $C_1^n < \dots< C_{N-1}^n$ be the indices $i$ for which $R_i = 1$, which are called the cut times. We also set $C_0^n = 0$ and $C_k^n = n$ for $k \ge N$. For fixed $f$, we often write $\T_n$ instead of $\T_n^f$. The tree $\T_n$ is constructed branch by branch where, for $j=1,\dots, N$, the $j$th branch, also called stick, is a path consisting of the vertices $C_{j-1}^n + 1, C_{j-1}^n + 2, \dots, C_j^n$. For $j \in [N-1]$, let $B_j^n$ be the random vertex that vertex $C_j^n + 1$ attaches to. We call $B^n_j$ the $j$th attachment point. Observe that conditional on $C_1^n, \dots, C_{N-1}^n$, we have that the attachment points are independent and $B_i^n \sim \text{Unif}([C_i^n])$. 

From the above description, we see that $\T_n$ is constructed branch by branch from a total of $N$ branches and we iteratively attach branch $j+1$ of length $C_j^n-C_{j-1}^n$ to the already constructed tree at the uniformly chosen vertex with label $B_j^n$. To introduce some notation, we let $\T_n(a)$ denote the subgraph of $\T_n$ on the labels $[a]$ with the convention that $\T_n(a) = \T_n$ for $a \ge n$. Let $\T_n^{(j)} = \T_n(C_j^n)$ denote the tree constructed from gluing the first $j$ branches. Lastly, we let $\nu_n^{(j)}$ denote the uniform vertex measure on $\T_n^{(j)}$ and $\nu_n$ the uniform vertex measure on $\T_n$ so that $\nu_n^{(N)} = \nu_n$ with the convention that $\nu^{(k)}_n = \nu_n$ for $k >N$.

Throughout this paper, we work with two choices of $f$. In Section \ref{Section: the polynomial case} we work with $f(n,i) = (i/n)^\beta$ for $\beta>0$ and the corresponding tree is denoted $\T_n^f$. In Section \ref{Section: the logarithmic case}, we work with $g(n,i) = \ln^\gamma(in^{-\frac{1}{2}} + 1)n^{-\frac{1}{2}}$ and denote the corresponding tree $\T_n^g$.

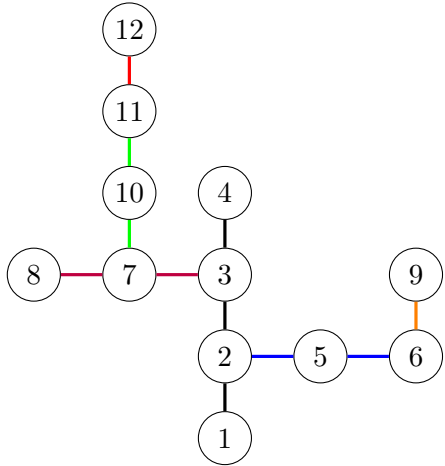
\begin{figure}[h!]
\centering
\begin{tikzpicture}[
every node/.style={circle,draw=black,fill=white,minimum size=7mm,inner sep=1pt},
stickone/.style={draw=black,very thick},
sticktwo/.style={draw=blue,very thick},
stickthree/.style={draw=purple,very thick},
stickfour/.style={draw=orange,very thick},
stickfive/.style={draw=green,very thick},
sticksix/.style={draw=red,very thick},
scale=0.9
]

\node (1) at (0,0) {1};
\node (2) at (0,1.2) {2};
\node (3) at (0,2.4) {3};
\node (4) at (0,3.6) {4};
\node (5) at (1.4,1.2) {5};
\node (6) at (2.8,1.2) {6};
\node (7) at (-1.4,2.4) {7};
\node (8) at (-2.8,2.4) {8};
\node (9) at (2.8,2.4) {9};
\node (10) at (-1.4,3.6) {10};
\node (11) at (-1.4,4.8) {11};
\node (12) at (-1.4,6.0) {12};

\draw[stickone] (1)--(2)--(3)--(4);
\draw[sticktwo] (2)--(5)--(6);
\draw[stickthree] (3)--(7)--(8);
\draw[stickfour] (6)--(9);
\draw[stickfive] (7)--(10)--(11);
\draw[sticksix] (11)--(12);

\end{tikzpicture}
\caption{A sample $\T_{12}$ with $(R_1,\ldots,R_{11})=(0,0,0,1,0,1,0,1,1,0,1)$ and $(B_1^{12}, \dots, B_5^{12}) = (2,3,6,7,11)$. Hence $(C_1^{12}, \dots, C_5^{12}) = (4,6,8,9,11)$. The 6 sticks are indicated with different colors.} 
\label{fig:aggregation-tree-example}
\end{figure}

\subsection{The stick-breaking construction}
\label{Section: stick breaking}
We will define our continuous limit object as a random subspace of the sequence space $\ell^1$. Let $\eta$ be a PPP of intensity $f(t)dt$ on $\R_{\ge 0}$ for some Borel measurable and locally integrable $f:[0, \infty) \to [0, \infty)$ for which $\int_0^\infty f(t)dt = \infty$. Order the points of $\eta$ as $0 = C_0 < C_1 < \dots$ and conditional on $\eta$, let $B_1, B_2, \dots$ be independent with $B_i \sim \text{Unif}([0,C_i])$. Define $\rho(t):\R_{\ge 0} \to \ell^1$ piecewise on $(C_{i-1}, C_i]$ by
    $$\rho(t) = \begin{cases}
        \big(t,0,\dots,0) & \text{ for $t \in [0, C_1]$},\\
        \rho(B_{i-1})  +(t - C_{i-1})z_i  & \text{ for $t \in (C_{i-1}, C_i]$ with $i \ge 2$},
    \end{cases}$$
    where $z_i = (0,\dots,0,1,0,\dots)$ denotes the $i$th basis vector of $\ell^1$. Informally, stick $i$ `opens' a new coordinate direction $i$ in $\ell^1$ and extends distance $C_i - C_{i-1}$ into this dimension from $\rho(B_{i-1})$.) In correspondence with the definitions of $\T_n(a)$ and $\T_n^{(j)}$ above, define $\T(t) = \rho([0,t])$ and $\T^{(k)} = \rho([0,C_k])$. Let
    $$\T = \overline{\bigcup_{k=1}^\infty \T^{(k)}}$$
    be the closure of the union of partial trees $\T^{(k)}$ in $\ell^1$ and note that $\big(\T,d\big)$ is a random metric space, where $d$ denotes the restriction of the $\ell^1$ metric to $\T$. Let $\overline\lambda_k$ be the normalized Lebesgue measure on $[0,C_k]$ and let $\mu^{(k)} = \rho_* \overline{\lambda}$ be its push forward by $\rho$. We call $\rho(C_{i-1},C_i]$ the $i$th stick or $i$th branch of $\T$, and for $i \ge 2$, we call $B_{i-1}$ its attachment point.

    Throughout the paper, we work with two choices for the intensity function. In Section \ref{Section: the polynomial case}, where we prove Theorem \ref{Theorem: main result 1 (polynomial case)}, we work with $f(t) = t^\beta$ and the corresponding tree will be denoted $\T_\beta$. In Section \ref{Section: the logarithmic case}, where we prove Theorem \ref{thm:log}, we work with $g(t) = \ln^\gamma(t+1)$ and the corresponding tree is denoted $\T_\gamma$.

    \subsection{The Gromov--Hausdorff-Prokhorov Topology}
    \label{Section: GHP topology}
    Let $(X,d_X)$ be a metric space. For $A \subset X$, we let $A^\epsilon := \{x \in X : d_X(x, A) \le \epsilon\}$ be the $\epsilon$ thickening of $A$. For $A,B \subset X$, we say,
    $$d_H^X(A,B) = \inf\{\epsilon : B \subset A^\epsilon \text{ and } A \subset B^\epsilon\},$$
    is the Hausdorff distance between $A$ and $B$ and simply write $d_H(A,B)$ if the underlying metric space $X$ is clear from context. Let $\mathcal{P}(X)$ denote the set of Borel probability measures on $(X, \mathcal{B}(X))$. For $\mu, \nu \in \mathcal{P}(X)$, let
     $$d_{P}(\mu, \nu) = \inf\{\epsilon > 0 : \mu(A) \le \nu(A^{\epsilon}) + \epsilon \text{ and } \nu(A) \le \mu(A^\epsilon) + \epsilon\text{ for all $A \in \mathcal{B}(X)$}\},$$
    denote the Prokhorov metric. The metric space $(\mathcal{P}(X), d_P)$ is complete whenever $X$ is complete. We call $(X,d_X,\mu)$ for $\mu \in \mathcal{P}(X)$ a measure metric space and say two measure metric spaces $(X, d_X, \mu_X)$ and $(Y, d_Y, \mu_Y)$ are equivalent if there exists a bijective isometry $\phi: X \to Y$ such that $\phi_* \mu_X = \mu_Y$ where $\phi_*\mu$ denotes the pushforward of $\mu$ under $\phi$. With $\mathbbm{M}$ we denote the set of compact measure metric spaces modulo the above equivalence relation and denote elements in $\mathbbm{M}$ by elements in the equivalence class. On $\mathbbm{M}$ we define the Gromov--Hausdorff--Prokhorov (GHP) metric as,
    $$d_{GHP}(X, Y) = \inf_{\varphi, \psi, Z} \left\{ \max \left( d_H^Z(\varphi(X), \psi(Y)), d_P^Z(\mu_X \circ \varphi^{-1},  \mu_Y \circ \psi^{-1}) \right) \right\},$$
    where the infimum is over all metric spaces $Z$ and isometric embeddings $\varphi: X \to Z$, $\psi: Y \to Z$ and $d_H^Z, d_P^Z$ are the Hausdorff and Prokhorov distance in $Z$ respectively. Throughout this paper, we often work with an equivalent definition of the GHP-distance which is defined using correspondences.

    Let $(X, d_X, \mu_X), (Y, d_Y, \mu_Y) \in \mathbbm{M}$ be two measure metric spaces. We say $R \subset X \times Y$ is a correspondence between $X$ and $Y$ if for all $x \in X$ there exists at least one $y \in Y$ with $(x, y) \in R$, and vice versa. The distortion of a correspondence $R$ is defined as
    $$\operatorname{dis}(R) = \sup_{(x, y), (x', y') \in R} \left| d_X(x, x') - d_Y(y, y') \right|.$$
    Lastly, let $\pi$ be a measure on the product space $(X \times Y)$ and $p_i$ be the projection onto $i$th coordinate for $i = 1,2$. We define the discrepancy of $\pi$ with respect to $\mu_X$ and $\mu_Y$ as
    $$D(\pi; \mu_X, \mu_Y) = ||\pi \circ p_1^{-1} - \mu_X||_{TV} + ||\pi \circ p_2^{-1} - \mu_Y||_{TV},$$
    where $||\mu - \nu||_{TV} = \sup\{|\mu(A) - \nu(A)| : A \subset X \text{ measurable}\}$ for $\mu,\nu$ measures on a measurable space $(X, \mathcal{F}_X)$ is the total variation distance. With these definitions, we can rephrase the GHP-metric as,
    \begin{equation} \label{Equation: GHP topology}d_{GHP}(X,Y) = \inf_{R, \pi} \left\{\max\left(\frac{1}{2}\text{dis}(R), D(\pi; \mu_X, \mu_Y) + \pi(R^{c})\right)\right\},\end{equation}
        where the infimum is taken over all correspondences $R$ between $X$ and $Y$ and all measures on $X \times Y$, and $R^c$ denotes the complement of $R$ in $X \times Y$. For a proof of equivalence between both definitions, we refer to \cite[Theorem 3.6]{WorkGHPdistance}.

    \subsection{Proof outline}
    \label{Section: proof outline}
    We only outline the proof strategy for determining the scaling limit of $\T_n^f$ for $f(n,i) = (i/n)^\beta$, for which we largely adapt Aldous' seminal proof for convergence of a uniformly random labelled tree to the CRT \cite{AldousCRT1}. We first aim to prove
    $$n^{-\frac{\beta}{\beta+1}}\big(C_1^n, \dots, C_k^n, B_1^n, \dots, B_k^n\big) \xrightarrow[n\to\infty]{d}\big(C_1, \dots, C_k, B_1, \dots, B_k\big),$$
    where $C_1 < \dots < C_k$ are the first $k$ points of a PPP of intensity $t^\beta dt$ on $\R_{\ge 0}$, which is shown in Section \ref{Section: convergence of branches polynomial case}. This tells us that the lengths and attachment locations of the branches used in the construction of $\T_n^f$ converge, after appropriate rescaling, to the stick-breaking law of $\T_\beta$. This leads to the following result.
    \begin{equation}
    \label{Equation: finite dimensional distribution}
        \big(\T_n^{(k)}, n^{-\frac{\beta}{\beta+1}}d_n, \nu_n^{(k)}\big) \xrightarrow[n \to \infty]{d} \big(\T_\beta^{(k)}, d, \mu^{(k)}\big),
    \end{equation}
    convergence being in the Gromov--Hausdorff--Prokhorov topology. This result is covered in Section \ref{Section: GH convergence polynomial case}. This shows that as $n \to \infty$, the discrete tree constructed from the first $k$ branches together with uniform measure $\nu_n^{(k)}$ and $\T_\beta^{(k)}$ together with $\mu^{(k)}$ converge in distribution.
    
    To show this holds for the full trees, we have to show that adding the remaining branches does not change the discrete and continuum trees much. Concretely, by \cite[Theorem 4.2]{billingsley1968convergence}, it suffices to show that for all $\epsilon>0$
    \begin{align*}&\lim_{k \to \infty}\pr\big(d_{GHP}\big((\T^{(k)}_\beta, d, \mu^{(k)}), (\T_\beta, d, \mu)\big)>\epsilon\big) = 0,\\
    &\lim_{k \to \infty} \limsup_{n \to \infty}\pr\big(d_{GHP}\big((\T^{(k)}_n, n^{-\frac{\beta}{\beta+1}}d_n, \nu_n^{(k)}), (\T_n^f, n^{-\frac{\beta}{\beta+1}}d_n, \nu_n)\big)>\epsilon\big) = 0.
    \end{align*}
    Since $\T_\beta^{(k)} \subset \T$ and $\T_n^{(k)} \subset \T_n^f$ as measurable metric spaces, we can turn Gromov--Hausdorff--Prokhorov convergence into separate statements on Hausdorff convergence and Prokhorov convergence. Thus Theorem \ref{Theorem: main result 1 (polynomial case)} is shown upon showing the result in \eqref{Equation: finite dimensional distribution} together with the four statements below, for all $\epsilon>0$.
    \begin{enumerate}[i)]
    \item  $\lim_{t\to \infty}\pr\!\left(d_H\left(\T(t), \T\right)>\epsilon\right) = 0;$  \hfill Section \ref{Section: compactness T polynomial case}
    \item  $\lim_{t\to \infty}\limsup_{n \to \infty} \pr\big(d_H\big(\T_n(tn^\frac{\beta}{\beta+1}), \T_n\big)>\epsilon n^\frac{\beta}{\beta+1}\big) = 0;$  \hfill Section \ref{Section: compactness T_n polynomial case}
    \item  $\lim_{k\to \infty}\pr\!\left(d_P\left(\mu^{(k)}, \mu\right)>\epsilon)\right) = 0;$ \hfill Section \ref{Section: convergence measures continuous polynomial case}
    \item  $\lim_{k\to \infty}\limsup_{n \to \infty} \pr\!\left(d_P\left(\nu_n^{(k)}, \nu_n\right)>\epsilon\right) = 0;$ \hfill Section \ref{Section: convergence measures discrete polynomial case}
\end{enumerate}
where in $i)$, we replaced $d_H(\T^{(k)}, \T)$ with $d_H(\T(t), \T)$ as the number of branches of $\T(t)$ is almost surely finite for any $t > 0$. the analogous replacement in ii) follows from the convergence of the rescaled discrete cut times

The proof for finding the scaling limit of $\T_n^g$ where $g(n,i) = \ln^\gamma(in^{-\frac{1}{2}} + 1)n^{-\frac{1}{2}}$ follows the exact same steps as outlined above, with the adaptations of \eqref{Equation: finite dimensional distribution}, i), ii), iii) and iv) proved in Sections \ref{Section: convergence first branches logarithmic case}, \ref{Section: compactness of T logarithmic setting}, \ref{Section: Compactness of T_n logarithmic setting}, \ref{Section: convergence measures continuous logarithmic case} and \ref{Section: convergence measures continuous logarithmic case} respectively. 

\section{The polynomial case}
\label{Section: the polynomial case}

Recall that $\T_{n}^f$ can be constructed by iteratively attaching branches of length $C_j^n - C_{j-1}^n$ where the $C_j^n$'s are the indices for which $R_i = 1$ with $R_i \sim \text{Ber}(f(n,i))$ for $f(n,i) = (i/n)^\beta$. The branches are attached to the previously built tree at the vertex with label $B_i^n$ for $B_i^n \sim \text{Unif}([C_i^n])$. Throughout this section, we let $C_1 < C_2 < \dots$ denote the ordered points of a PPP of intensity $t^\beta dt$ on $\R_{\ge 0}$ and conditional on the $C_j$'s let $B_i \sim \text{Unif}([0,C_i])$ independently.
\subsection{Convergence of branch and attachment points -- polynomial case}
\label{Section: convergence of branches polynomial case}
The goal of the current section is proving the following convergence in distribution.
\begin{prop}
\label{Proposition: Convergence repeat/attachment points polynomial case}
    For all $k \in \N$, we have,
    $$n^{-\frac{\beta}{\beta+1}}\big(C_1^n, \dots, C_k^n, B^n_1, \dots, B^n_k\big) \xrightarrow[n \to \infty]{d} \big(C_1, \dots, C_k, B_1, \dots, B_k\big).$$
\end{prop} 
In order to prove the above result, it suffices to show,
\begin{align*}
        &i) \;\; \pr\big(C_1^n \le s_1n^\frac{\beta}{\beta+1}, \dots, C_k^n \le s_k n^\frac{\beta}{\beta+1}\big) \xrightarrow[n \to \infty]{}\pr\big(C_1 \le s_1, \dots, C_k \le s_k\big),\\
        &ii) \; \pr\Big(B_1^n\le t_1C^n_1,\dots, B_k^n \le t_kC^n_k \; \Big|\; C_1^n\le s_1n^{\frac{\beta}{\beta+1}}, \dots, C_k^n\le s_{k}n^{\frac{\beta}{\beta+1}}\Big),\\
        &\xrightarrow[n \to \infty]{} \pr\Big(B_1\le t_1C_1,\dots, B_k\le t_kC_k \; \Big|\; C_1\le s_1, \dots, C_k\le s_{k}\Big).
    \end{align*}
    for $0 < s_1 < \dots < s_k$ and $t_1, \dots, t_k \in [0,1)$. We prove these two statements in the following four lemmas.

\begin{lem}
\label{Lemma: convergence to pdf first points PPP}
    For all $0 < s_1 < \dots < s_k$, we have
    $$n^{\frac{k\beta}{\beta+1}}\:\pr\!\left(C_1^n = \floor{s_1n^\frac{\beta}{\beta+1}}, \dots, C_k^n = \floor{s_kn^\frac{\beta}{\beta+1}}\right) \xrightarrow[n \to \infty]{\text{u.c.}} g_{C_1, \dots, C_k}(s_1, \dots, s_k),$$
    where u.c. denotes uniform convergence over compact sets and $g_{C_1, \dots, C_k}(s_1, \dots, s_k)$ is the joint density of the first $k$ ordered points of a PPP of intensity $t^\beta dt$.
\end{lem}

\begin{proof}
    Recall that $g_{C_1, \dots, C_k}(s_1, \dots, s_k) = s_1^\beta \dots s_k^\beta e^{-\frac{s_k^{\beta + 1}}{\beta + 1}}$ is the density of the first $k$ ordered points of a PPP of intensity $t^\beta dt$ on $\R_{\ge 0}$. Also $\big\{C_1^n = \floor{s_1n^\frac{\beta}{\beta+1}}, \dots, C_k^n = \floor{s_kn^\frac{\beta}{\beta+1}}\big\}$ happens precisely when the Bernoulli random variables $R_i$ satisfy
    $$R_i = \begin{cases}
        1 &\text{ if } i \in \big\{\floor{s_1n^\frac{\beta}{\beta+1}}, \dots, \floor{s_kn^\frac{\beta}{\beta+1}}\big\}\\
        0 &\text{ if } i \in [\floor{s_k n^\frac{\beta}{\beta+1}}] \setminus \big\{\floor{s_1n^\frac{\beta}{\beta+1}}, \dots, \floor{s_kn^\frac{\beta}{\beta+1}}\big\}.
    \end{cases}$$
    Hence we obtain,
    \begin{align}
    \label{Equation: computations cut points converge to pdf PPP points}
        n^{\frac{k\beta}{\beta+1}}\:\pr\!&\left(C_1^n = \floor{s_1n^\frac{\beta}{\beta+1}}, \dots, C_k^n = \floor{s_kn^\frac{\beta}{\beta+1}}\right)\!,\notag\\
        &= n^{\frac{k\beta}{\beta+1}}\prod_{i=1}^k \frac{f\big(n, \floor{s_in^\frac{\beta}{\beta+1}}\big)}{1 - f\big(n, \floor{s_in^\frac{\beta}{\beta+1}}\big)} \hspace{-6pt}\prod_{i=1}^{\floor{s_kn^\frac{\beta}{\beta+1}}}\hspace{-6pt}(1 - f(n,i)),\notag\\
        &= \prod_{i=1}^k \left(n^{\frac{\beta}{\beta+1}}\frac{\floor{s_in^\frac{\beta}{\beta+1}}}{n^\frac{\beta}{\beta+1}}\right)^\beta \prod_{i=1}^{\floor{s_kn^\frac{\beta}{\beta+1}}} \left(1 - \left(\frac{i}{n}\right)^\beta \right)\prod_{i=1}^k \left(1 - \left(\frac{\floor{s_in^\frac{\beta}{\beta+1}}}{n}\right)^\beta\right)^{-1},\notag\\
        &=(1 + o(1))s_1^\beta \cdots s_k^\beta \prod_{i=1}^{\floor{s_kn^\frac{\beta}{\beta+1}}} \left(1 - \left(\frac{i}{n}\right)^\beta\right)
    \end{align}
    where the above convergence is uniform over compact sets. Thus it only remains to show
    $$\prod_{i=1}^{\floor{s_kn^\frac{\beta}{\beta+1}}}\left(1 - \left(\frac{i}{n}\right)^\beta \right)\xrightarrow[n \to \infty]{u.c.} e^{-\frac{s_k^{\beta + 1}}{\beta + 1}}.$$
    To this end, we take logarithms and use first-order Taylor expansion $\ln(1 - x) = -x + O(x^2)$ as $x \to 0$ where this expansion is uniform on compact sets. Using this, we obtain,
    \begin{align*}
        \prod_{i=1}^{\floor{s_kn^\frac{\beta}{\beta+1}}}\left(1 - \left(\frac{i}{n}\right)^\beta\right)&= \exp\left(\sum_{i = 1}^{\floor{s_kn^\frac{\beta}{\beta+1}}} \ln\left(1 - \left(\frac{i}{n}\right)^{\beta}\right)\right),\\
        &=\exp\left(-n^{-\beta}\sum_{i=1}^{\floor{s_kn^\frac{\beta}{\beta+1}}} i^\beta + O\left(n^{-2\beta}\sum_{i=1}^{\floor{s_kn^\frac{\beta}{\beta+1}}}i^{2\beta}\right)\right),\\
        &= \exp\left(-\frac{s_k^{\beta+1}}{\beta+1} + o(1)\right)
    \end{align*}
    convergence again being uniform over compact sets. This concludes the proof. 
\end{proof}

Before continuing with the proof of Proposition \ref{Proposition: Convergence repeat/attachment points polynomial case}, we state the following elementary result.

\begin{lem}
\label{uniform on compact sets and integrals}
    Let $K \subset \R^k$ be compact and $A,A_n \subset K$ be Borel measurable sets. Suppose $\lambda(A_n \Delta A) \to 0$ as $n \to \infty$, for $\lambda$ the Lebesgue measure and that $g_n: K \to \R$ is measurable and converges uniformly on $K$ to measurable and integrable $g: K \to \R$. Then,
    $$\int_{A_n} g_n(x) dx \xrightarrow[n \to \infty]{} \int_{A} g(x)dx.$$
\end{lem}
Using the above Lemma, we can proof $i)$ of Proposition \ref{Proposition: Convergence repeat/attachment points polynomial case}.
\begin{lem}
\label{convergence of repeat probabilities}
For $0 < s_1 < \dots < s_k$, we have,
    $$\pr\Big(C_1^n \le s_1n^{\frac{\beta}{\beta+1}}, \dots, C_k^n \le s_{k}n^{\frac{\beta}{\beta+1}}\Big) \xrightarrow[n \to \infty]{} \pr\big(C_1 \le s_1, \dots, C_k \le s_k\big).$$
\end{lem}

\begin{proof}
\label{proof in which we go from sum to integral}
    Write $s_i^n = (\floor{s_i n^\frac{\beta}{\beta+1}} + 1)n^{-\frac{\beta}{\beta+1}}$ and $y_i^n = (\floor{y_i n^\frac{\beta}{\beta+1}} + 1)n^{-\frac{\beta}{\beta+1}}$. Using Lemma \ref{uniform on compact sets and integrals} together with Lemma \ref{Lemma: convergence to pdf first points PPP}, we get,
\begin{align*}
\label{sums converging to integrals}
    \pr\big(&C_1^n \le s_1n^{\frac{\beta}{\beta+1}}, \dots, C_k^n \le s_kn^{\frac{\beta}{\beta+1}}\big) =\hspace{-5pt} \sum_{x_1=1}^{\floor{s_1n^{\frac{\beta}{\beta+1}}}}\sum_{x_2 = x_1+1}^{\floor{s_2n^{\frac{\beta}{\beta+1}}}} \hspace{-4pt}\dots\hspace{-4pt} \sum_{x_k = x_{k-1}+1}^{\floor{s_kn^{\frac{\beta}{\beta+1}}}} \hspace{-3pt}\pr\big(C_1^n = x_1, \dots C_k^n = x_k\big),\\
    &=\int_{1}^{\floor{s_1n^{\frac{\beta}{\beta+1}}}+1}\hspace{-2pt}\int_{\floor{x_1} + 1}^{\floor{s_2n^{\frac{\beta}{\beta+1}}}+1} \hspace{-2pt} \dots \int_{\floor{x_{k-1}}+1}^{\floor{s_kn^{\frac{\beta}{\beta+1}}}+1} \pr\big(C_1^n = \floor{x_1}, \dots C_k^n = \floor{x_k}\big)dx_k \dots dx_2dx_1,\\
    &=\int_{n^{-{\frac{\beta}{\beta+1}}}}^{s_1^n}\hspace{-3pt}\int_{y_1^n}^{s_2^n} \hspace{-3pt}\dots \int_{y_{k-1}^n}^{s_k^n} n^\frac{\beta k}{\beta + 1}\pr\big(C_1^n = \floor{y_1n^{\frac{\beta}{\beta+1}}}, \dots, C_k^n = \floor{y_kn^{\frac{\beta}{\beta+1}}}\big)dy_k\dots dy_2dy_1,\\
    &\xrightarrow[n \to \infty]{}\int_0^{s_1} \int_{y_1}^{s_2} \dots \int_{y_{k-1}}^{s_k} y_1^\beta \dots y_k^\beta e^{-\frac{y_k^{\beta+1}}{\beta+1}}dy_k \dots dy_2 dy_1,\\
    &= \pr(C_1 \le s_1, \dots, C_k \le s_k),
\end{align*}
where we used the substitution $x_i = n^{\frac{\beta}{\beta+1}}y_i$ and recognize $y_1^\beta \cdots y_k^\beta e^{-\frac{y_k^{\beta+1}}{\beta+1}}$ as the joint density of the first~$k$ ordered points of a PPP of intensity $t^\beta dt$.
\end{proof}

\begin{lem}
\label{Lemma: attachment points are uniform}
For $t_1, \dots, t_k \in [0,1)$ and $0<s_1<\dots<s_k$, we have,
\begin{align*}
\pr\big(B_1^n\le t_1C^n_1,\dots, &B_k^n\le t_kC^n_k \; \big|\; C_1^n \le s_1n^{\frac{\beta}{\beta+1}}, \dots, C_k^n\le s_kn^{\frac{\beta}{\beta+1}}\Big)\\
\xrightarrow[n \to \infty]{} &\pr\big(B_1\le t_1C_1,\dots, B_k\le t_kC_k \; \big|\; C_1 \le s_1, \dots, C_k \le s_k\big).
\end{align*}
\end{lem}
\begin{proof}
Write $\mathcal{C}_k^n(s) = \big\{C_1^n \le s_1n^\frac{\beta}{\beta+1}, \dots, C_k^n \le s_kn^\frac{\beta}{\beta+1}\big\}$ where $s = (s_1, \dots, s_k)$. Since we have $\pr\big(B_1\le t_1C_1,\dots, B_k\le t_kC_k \; \big|\; C_1 \le s_1, \dots, C_k \le s_k\big) = t_1 \cdots t_k$, it suffices to show $\pr\big(B_1^n \le t_1C_1^n, \dots, B_k^n \le t_kC_k^n \; \big| \; \mathcal{C}_k^n(s)\big) \xrightarrow[n \to \infty]{} t_1 \cdots t_k$. Using the tower property of expectation, we obtain,
\begin{align*}
    \pr\big(B_1^n \le t_1C_1^n, \dots, B_k^n \le t_kC_k^n \; \big| \; \mathcal{C}_k^n(s)\big) &= \E\!\!\left[\prod_{i=1}^k \frac{\floor{t_iC_i^n}}{C_i^n} \; \Big| \; \mathcal{C}_k^n(s)\right]\!\!,\\
    &= \E\!\big[t_1 \cdots t_k + O(1/C_1^n) \; \big| \; \mathcal{C}_k^n(s)\big].
\end{align*}
It remains to show $C_1^n$ diverges in probability. For this, let $M>0$ be fixed. Then,
$$\pr(C_1^n \le M\; | \; \mathcal{C}_k^n(s)) \le \pr(C_1^n \le M)/\pr(\mathcal{C}_k^n(s)) = O\left(\sum_{i=1}^M f(n,i)\right) = O(n^{-\beta})$$
and hence the result is shown.
\end{proof}

\begin{proof}[proof of Proposition \ref{Proposition: Convergence repeat/attachment points polynomial case}]
    Follows directly from Lemma \ref{Lemma: attachment points are uniform} together with Lemma \ref{convergence of repeat probabilities}.
\end{proof}

\subsection{Gromov--Hausdorff--Prokhorov convergence}
\label{Section: GH convergence polynomial case}
We have shown that the law of the first $k$ cut points of the branches in $\T_n^f$ converge to that of the first $k$ points of a Poisson process of intensity $t^\beta dt$ on $\R_{\ge 0}$. Hence, the length of the first $k$ sticks of $\T_n$ and $\T$ share the same distribution after rescaling. Since in both cases the sticks get attached uniformly over the already constructed tree, we should expect the trees to look the same. This is formalized in the statement below, which will be proved in this section.

\begin{thm}
\label{Theorem: finite tree convergence}
    We have $\big(\T_n^{(k)}, n^{-{\frac{\beta}{\beta+1}}}d_n, \nu_n^{(k)}\big) \xrightarrow[n \to \infty]{d} \big(\T^{(k)}, d, \mu^{(k)}\big)$ in the GHP-topology.
\end{thm}

Firstly, we work in a probability space where $n^{-\frac{\beta}{\beta+1}}(C_i^n, B_i^n) \xrightarrow[n\to\infty]{}(C_i,B_i)$ holds almost surely and we also assume $B_i \neq C_j$ for all $i,j \in [k]$ which holds with probability $1$. Hence it suffices to show that for any sequence $n^{-\frac{\beta}{\beta+1}}(C_i^n, B_i^n) \to (C_i, B_i)$ with $B_i \neq C_j$ we have that the corresponding trees $\T_n^{(k)}$ and $\T^{(k)}$ and measures converge in distribution in the GHP-topology. Below, we define a correspondence $R_n$ between the trees $\T_n$ and $\T$ with vanishing distortion and find measure $\pi_n: \T_n^{(k)} \times \T^{(k)} \to [0,1]$ with vanishing discrepancy $D(\pi_n; \nu_n^{(k)}, \mu^{(k)})$ for which $\pi_n(R^c_n) \to 0$. This ensures convergences in the GHP topology as seen in Equation \eqref{Equation: GHP topology}.

\begin{definition}
\label{definition: correspondence for convergence tree}
   Define $\phi_n:[0, C_k] \to \T_n^{(k)}$ as $\phi_n(0) = 1$ and
    $$\phi_n(x) = \left\lceil C^n_{i-1} + \frac{x - C_{i-1}}{C_i - C_{i-1}}\big(C_i^n - C_{i-1}^n\big) \right\rceil \qquad \text{for $x \in (C_{i-1}, C_i]$},$$
    to be the projection of the first $k$ sticks of $\T$ (as subset of $\R$) onto the first $k$ sticks of $\T_n^{(k)}$. We define correspondence $R_n^{(k)} \subset \big(\T^{(k)} \times \T_n^{(k)}\big)$ as,
    $$R_n^{(k)} = \big\{\big(\rho(x), \phi_n(x)\big) \text{ for $x \in [0, C_k)$}\big\}.$$
\end{definition}
\begin{lem}
\label{relation (given almost sure event) has discrepency 0}
Suppose $n^{-\frac{\beta}{\beta+1}}(C^n_i, B^n_i) \to (C_i, B_i)$ as $n \to \infty$ with $C_i \neq B_j$ for $i,j \in [k]$. Then,
$$\lim_{n \to \infty}\operatorname{Dis}(R_n) = 0.$$
\end{lem}
\begin{proof}
    To simplify notation, write $\Delta_n(x,y) = d(\rho(x), \rho(y)) - n^{-\frac{\beta}{\beta+1}}d_n(\phi_n(x), \phi_n(y))$ where $x,y \in \R_{\ge 0}$. We apply induction on $k$. For $k = 1$ the statement is true. Indeed suppose $x,y \in [0,C_1]$ with $x>y$. Then we have,
    \begin{align*}
        \Delta_n(x,y) &= x - y - n^{-\frac{\beta}{\beta+1}}\left(\left\lceil \frac{x}{C_1}C_1^n\right\rceil - \left\lceil \frac{y}{C_1}C_1^n\right\rceil\right),\\
        &= (x-y)\left(1 - n^{-\frac{\beta}{\beta+1}}\frac{C_1^n}{C_1}\right) + \delta n^{-\frac{\beta}{\beta+1}},
    \end{align*}
    for some $\delta \in (-1, 1)$. This expression vanishes as $n \to \infty$ since $n^{-\frac{\beta}{\beta+1}} C_1^n \to C_1$. Since this holds uniformly for all $x,y \in [0,C_1]$ we conclude $\text{Dis}(R_n^{(1)}) \to 0$ as $n \to \infty$.
    
    Next suppose $\text{Dis}(R_n^{(k-1)})$ vanishes as $n \to \infty$. We prove the analogous result for $R_n^{(k)}$. Suppose $x,y \in [0,C_k]$, we aim to show $\Delta_n(x,y) \to 0$ uniformly in $x,y$. If $x,y$ are both in $[0,C_{k-1}]$, the induction hypothesis immediately applies. If both $x,y \in (C_{k-1}, C_k]$, an analogous proof to the base case suffices. Thus, suppose $x \in (C_{k-1}, C_k]$ while $y \in [0,C_{k-1}]$. Then we have,
    \begin{align*}|\Delta_n(x,y)| &\le |d(\rho(x), \rho(B_{k-1})) - n^{-\frac{\beta}{\beta+1}}d_n(\phi_n(x), B_{k-1}^n)|\\
    &+ n^{-\frac{\beta}{\beta+1}}d_n(\phi_n(B_{k-1}), B_{k-1}^n)+ |\Delta_n(B_{k-1}, y)| \end{align*}
I.e. we decompose the error $|\Delta_n(x,y)|$ into the error of the new branch, the error between attachment points and the error between $\T^{(k-1)}$ and $\T_n^{(k-1)}$. The final term vanishes by the induction hypothesis. For the second term, recall that $B_{k-1} \neq C_i$ for $i \in [k]$ so that $\rho(B_{k-1})$ and vertex $B_{k-1}^n$ must lie on the same branch for $n$ large enough. Since $n^{-\frac{\beta}{\beta+1}}B_{k-1}^n \to B_{k-1}$ as $n \to \infty$, this term also vanishes. For the first term, we have,
$$d_n(\phi_n(x), B_{k-1}^n) = \left\lceil C_{k-1}^n + \frac{x - C_{k-1}}{C_k - C_{k-1}}(C_k^n - C_{k-1}^n)\right\rceil - C_{k-1}^n,$$
and thus,
\begin{align*}
    |d(\rho(x), \rho(B_{k-1})) &- n^{-\frac{\beta}{\beta+1}}d_n(\phi_n(x), B_{k-1}^n)|,\\
    &\le \left|x - C_{k-1} - n^{-\frac{\beta}{\beta+1}}\left(\frac{x-C_{k-1}}{C_k-C_{k-1}}(C_k^n - C_{k-1}^n) + \delta\right)\right| \xrightarrow[n \to \infty]{} 0
\end{align*}
where $\delta \in (-1,1)$. The above proof uniformly bounds $|\Delta_n(x,y)|$ for all $x, y \in [0,C_k]$, so the desired result follows. 
\end{proof}
We now define a measure $\pi_n: \T_n^{(k)} \times \T^{(k)} \to [0,1]$ for which $\pi_n\big(\big(R_n^{(k)}\big)^c\big) \to 0$ and $D(\pi_n ; \nu_n^{(k)}, \mu^{(k)}) \to 0$  as $n \to \infty$. 
\begin{definition}
    For measurable $A \subset \T_n^{(k)}, B \subset \T^{(k)}$, and $\phi_n$ as in Definition \ref{definition: correspondence for convergence tree}, set
$$\pi_n(A,B) = \mu^{(k)}\big(\big\{x \in B : \phi_n(\rho^{-1}(x)) \in A\big\}\big).$$
\end{definition}
\begin{lem}
    With $R_n$ as in Definition \ref{definition: correspondence for convergence tree}, we have for $k \in \N$,
    $$i) \; \pi_n\big(\big(R_n^{(k)}\big)^c\big) = 0 \quad \text{ and } \quad ii) \;D\big(\pi; \nu_n^{(k)}, \mu^{(k)}\big) \xrightarrow[n \to \infty]{} 0.$$
\end{lem}

\begin{proof}
$i)$ is immediate as $\pi_n(R_n^c) = \mu^{(k)}(\emptyset) = 0, \text{ for all } n.$
We continue with $ii)$. For this, let $B \subset \T^{(k)}$ and let $p_i$ denote projection onto $i$th coordinate for $i = 1,2$. Then, 
$$\pi_n(\T_n^{(k)}, B) = \mu^{(k)}\big(\big\{x \in B : \phi_n(\rho^{-1}(x)) \in \T_n^{(k)}\big\}\big) = \mu^{(k)}(B),$$
and thus $||\pi_n \circ p_2^{-1} - \mu^{(k)}||_{TV} = 0$ for all $n$. It remains to show that $||\pi_n \circ p_1^{-1} - \nu_n^{(k)}||_{TV} \to 0$ as~$n \to \infty$. For this, let $A \subset \T_n^{(k)}$. We compute,
\begin{align*}
\pi_n\big(A, \T^{(k)}\big) &= \mu^{(k)}\big(\big\{x \in \T^{(k)} : \phi_n(\rho^{-1}(x)) \in A\big\}\big),\\
&= \sum_{a \in A} \mu^{(k)}\big(\big\{x \in \T^{(k)} : \phi_n(\rho^{-1}(x)) = a\big\}\big).
\end{align*}
Let $s \in \{C_{i-1}^n + 1, \dots, C_i^n\}$ be a vertex of $\T_n^{(k)}$ on the $i$th branch and recall that $\mu^{(k)} = \rho_* \bar \lambda$ is the pushforward by $\rho$ of the normalized Lebesgue measure on $[0,C_k]$. Hence,
\begin{align*}
    \mu^{(k)}\big(\big\{x \in \T^{(k)} : \phi_n(\rho^{-1}(x)) = s\big\}\big) &= \bar\lambda\big(\big\{x \in [0,C_k] : \phi_n(x) = s\big\}\big),\\
    &=\bar \lambda\Big(\Big\{x : C_{i-1}^n + \frac{x - C_{i-1}}{C_i - C_{i-1}}(C_i^n - C_{i-1}^n) \in (s-1, s]\Big\}\Big),\\
    &= \frac{1}{C_k}\frac{C_i - C_{i-1}}{C_i^n - C_{i-1}^n},\\
    &= \frac{1}{C_k^n} + o\big(n^{-\frac{\beta}{\beta+1}}\big).
\end{align*}
We conclude,
$$\pi_n(A, \T^{(k)}) = \sum_{s \in A} \mu^{(k)}\big(\big\{x \in \T^{(k)} : \phi_n(x) = s\big\}\big) = \frac{\#A}{C_k^n} + \#A \times o\big(n^{-\frac{\beta}{\beta+1}}\big) = \nu_n^{(k)}(A) + o(1),$$
since $\frac{\#A}{C_k^n} = \nu_n^{(k)}(A)$ and $\#A \le C_k^n = O\big(n^\frac{\beta}{\beta+1}\big)$. We conclude $D\big(\pi_n ; \nu^{(k)}_n, \mu^{(k)}\big) \xrightarrow[n \to \infty]{} 0$ which verifies $ii)$ and finishes the proof.
\end{proof}

\subsection{Hausdorff approximation of $\T$ by its first branches -- polynomial case}
\label{Section: compactness T polynomial case}
In this section, we aim to show,
\begin{thm}
\label{Theorem: limit continuous tree is 0}
    For all $\epsilon > 0$, we have,
    $$\lim_{t \to \infty} \pr\big(d_H\big(\T(t), \T\big)>\epsilon\big) = 0.$$
\end{thm}

From a union bound, it follows that it suffices to find $\epsilon_i : [0,\infty) \to [0, \infty)$ for which,
\begin{equation}\label{conditions on epsilon_i}i) \lim_{t \to \infty} \sum_{i=0}^\infty \epsilon_i(t) = 0 \quad \text{ and } \quad ii) \lim_{t \to \infty} \sum_{i=0}^\infty \pr\big(d_H\big(\T(2^it), \T(2^{i+1}t)\big) > \epsilon_i(t)\big) = 0.\end{equation}
We first work on the second statement, for which we introduce the following lemma.

\begin{lem}
\label{Lemma: distance point to T(a)}
    Fix $a,c>0$ and $s \in [a,2a]$. Then,
    $$\pr(d(\T(a), \rho(s))>c)\le 30\exp\left(-\frac{c a^\beta}{4}\right).$$
\end{lem}

\begin{proof}
We examine the ancestral line of $\rho(s)$ in $\T$ until its first intersection point with $\T(a)$.  Let $\eta$ be a PPP of intensity $t^\beta dt$ on $\R_{\ge 0}$ and let $U_1, U_2, \dots$ be i.i.d.\ $\text{Unif}([0,1])$ random variables independent of $\eta$. Let $t_1 = s$ and recursively define, 
    $$p(t_i) = \max\{\{x \in \eta : x \le t_i\} \cup \{a\}\}, \qquad d_i = t_i - p(t_i), \qquad t_{i+1} = \max\{a,U_ip(t_i)\},$$
Thus $p(t_i)$ corresponds to the left endpoint of the stick on which $\rho(t_i)$ lies or $p(t_i) = a$ if that left endpoint lies before $a$. Either way,  $d_i$ is the distance from $\rho(t_i)$ to the attachment point of that
stick or $\T(a)$. Lastly, we let $N=\min\{i\ge1:U_i p(t_i)\le a\}$ be the number of sticks traversed on the path from $\rho(s)$ to $\T(a)$. Observe that $t_{N+1} = a$ and hence $d_{j} = 0$ for $j \ge N+1.$

By construction, $d(\T(a), \rho(s)) = \sum_{j=1}^N d_j$. We first bound the tail of $N$. On the event $\{N > j\}$ we have $t_{i+1} = U_ip(t_i)$ for $i \le j$ and hence
$$a < t_{j+1} \le U_jp(t_j) \le U_jt_j \le U_jU_{j-1}t_{j-2}\le \cdots \le U_j \cdots U_1 s \le 2U_j \cdots U_1a.$$
This immediately implies
$$\pr(N > j) \le \pr(2aU_1 \cdots U_j > a) \le 2^x\left(\frac{1}{1+x}\right)^j$$
where the last step follows from a Chernoff bound and holds for all $x>0$. Next, we bound $\sum_{i=1}^j d_i$. First note that $\{d_i > u\}$ can only happen if $\eta([t_i - u, t_i]) = 0$. Formally, we see
\begin{align*}
    \pr(d_i > u \mid d_1, \dots, d_{i-1}, t_i) &= \mathbbm{1}_{\{t_i - u > a\}}\exp\left(-\int_{t_i - u}^{t_i} t^\beta dt\right) \le \exp(-u a^\beta).
\end{align*}
Thus the $d_i$s can be dominated by i.i.d.\ $\text{Exp}(\lambda)$ random variables with $\lambda=a^\beta$. This does not depend on $d_1,\dots, d_{i-1},t_i$ so we obtain
$$\pr\!\left(\sum_{i=1}^j d_i > c\right) \le e^{-c\theta} \left(\frac{\lambda}{\lambda-\theta}\right)^j$$
where the last bound again follows from a Chernoff bound and holds for $\theta \in (0, \lambda)$. To finish the argument, we apply a union bound to obtain that
$$\pr(d(\T(a), \rho(s))>c) \le \pr\!\left(\sum_{i=1}^j d_i > c\right) + \pr\!\left(N>j\right) \le e^{-c\theta}\left(\frac{\lambda}{\lambda - \theta}\right)^j + 2^x\left(\frac{1}{1+x}\right)^j,$$
where the above inequality holds for all $x>0, \theta \in (0,\lambda)$ and $j \in \Z_{\ge 0}$. By taking the values, $x = e - 1, \theta = \lambda\frac{e-1}{e}$ and $j = \floor{c\lambda\frac{e-1}{2e}}$, we obtain
$$\pr(d(\T(a), \rho(s))>c) \le \left( 1+e2^{e-1}\right) e^{-\lfloor \lambda c\frac{e-1}{2e}\rfloor}\le 30e^{-\frac{\lambda c}{4}},$$
as desired.
\end{proof}

Another way to see the above result is to sample the points $B_i$ and $C_i$ jointly: start with a homogeneous PPP in the wedge $\{(x,y): x\ge0,0\le y\le x^\beta\}$. The $x$-coordinates of the points are distributed as $C_i$ and the (rescaled) $y$-coordinates as $B_i$. A path from $s$ to $\T(a)$ of length exceeding $c$ requires an area $c a^\beta$ to be void of points, giving the exponential decay.  

The above bound on the distance between $\rho(s) \in \T(2a)$ and $\T(a)$ can be upgraded to a bound on $d_H(\T(a), \T(2a))$.
\begin{lem}
\label{Lemma: bound doubling tree}
    Let $a,c>0$. We have,
    $$\pr\big(d_H\big(\T(a), \T(2a)\big)>c\big) \le \frac{60a}{c} \exp\left(-\frac{c a^\beta}{8}\right)$$
\end{lem}
\begin{proof}
    By Lemma \ref{Lemma: distance point to T(a)}, we get,
    $$\mathbb{E}\!\left[\lambda\!\left( \left\{x \in [a,2a] : d\big(\T(a), \rho(x)\big)>\frac{c}{2}\right\}\right)\right]\le 30ae^{-\frac{ca^\beta}{8}}.$$
    Whenever $d_H(\T(a), \T(2a)) > c$, we must have $\lambda\left(\left\{x \in [a,2a] : d\big(\T(a), \rho(x)\big)>\frac{c}{2}\right\}\right) > \frac{c}{2}$. Indeed, if there is a point in $\T(2a)$ that is at distance $c$ away from $\T(a)$, then the portion of its ancestral line lying at distance at least $c/2$ from $\T(a)$ must have length at least $c/2$. Using this, we obtain the bound,
    $$\mathbb{E}\!\left[\lambda\!\left(\left\{x \in [a,2a] : d\big(\T(a), \rho(x)\big)>\frac{c}{2}\right\}\right)\right] \ge \frac{c}{2} \pr(d_H(\T(a), \T(2a)) > c).$$
    Combining both inequalities yields $\pr\big(d_H(\T(a), \T(2a))>c\big) \le \frac{60a}{c}e^{-\frac{ca^\beta}{8}}$.
\end{proof}

\begin{proof}[Proof of Theorem \ref{Theorem: limit continuous tree is 0}]
\label{proof continuous tree distance to 0}
Recall that it suffices to find $\epsilon_i: [0, \infty) \to [0, \infty)$ for which $i)$ and $ii)$ in Equation \eqref{conditions on epsilon_i} are satisfied. Substitute $a = 2^it$ and $c = \epsilon_i(t)$ in Lemma \ref{Lemma: bound doubling tree}. This yields,
$$\sum_{i=0}^\infty \pr\big(d_H\big(\T(2^it), \T(2^{i+1}t)\big)>\epsilon_i\big) \le 60 \sum_{i=0}^\infty\frac{2^{i}t}{\epsilon_i} \exp\left(-\frac{(2^{ i}t)^\beta\epsilon_i}{8}\right).$$
Take $\epsilon_i = \frac{8(i+1)(\beta+1)}{(2^it)^\beta}t^{\frac{\beta}{2}}$. Clearly, $i)$ of Equation \eqref{conditions on epsilon_i} is satisfied. We continue to verify $ii)$. Suppose $t$ is large enough that $t^\frac{\beta}{2} > \ln(2)$. Then,
\begin{align*}
    \sum_{i=0}^\infty \pr\big(d_H\big(\T(2^it), \T(2^{i+1}t)\big) > \epsilon_i\big) &\le 8\sum_{i=0}^\infty\frac{2^{i(\beta+1)}}{(i+1)(\beta+1)}t^{1 + \frac{\beta}{2}} \exp\left(-t^\frac{\beta}{2}(i+1)(\beta+1)\right)\\
    &\le 8t^{1+\frac{\beta}{2}} \sum_{i=0}^\infty \exp\left((1+\beta)(i+1)(\ln(2) - t^\frac{\beta}{2})\right)\\
    &\le 8t^{1+\frac{\beta}{2}} \sum_{i=0}^\infty \big(2\exp\!\big(\!-t^\frac{\beta}{2}\big)\big)^{i+1}\\
    &=8t^{1 + \frac{\beta}{2}}\frac{2\exp\!\big(\! - t^\frac{\beta}{2})}{1 - 2
    \exp\!\big(\!- t^\frac{\beta}{2}\!\big)\!}\xrightarrow[t \to \infty]{} 0,
\end{align*}
where in the last step, we recognized a geometric series. This concludes the proof. 
\end{proof}

\subsection{Hausdorff approximation of $\T_n$ by its first branches -- polynomial case}
\label{Section: compactness T_n polynomial case}
In this section, we aim to show,
\begin{thm}
\label{Theorem: compactness discrete tree uniform case}
For all $\epsilon > 0$, we have,
$$\lim_{t \to \infty} \limsup_{n \to \infty} \pr\!\left(d_H\left(\T_n(tn^\frac{\beta}{\beta+1}), \T_n\right)>\epsilon n^\frac{\beta}{\beta+1}\right) = 0.$$
\end{thm}
The proof follows a similar approach as Section \ref{Section: compactness T polynomial case}. We first introduce some notation. For $v \in \T_n$, let $p(v)$ denote the parent of $v$ with the convention that the root is its own parent. Iteratively, set $p^k(v) = p(p^{k-1}(v))$ with $p^0(v) = v$. We state and prove a bound on the probability that the distance between $\T_n(a)$ and some vertex $s \in \{a+1, \dots, 2a\}$ exceeds $c$, which is the discrete analogue of Lemma \ref{Lemma: distance point to T(a)}.

\begin{lem}
\label{Lemma: distance vertex tree discrete case}
    Fix $a \in \N$ and $s \in \{a+1, \dots, 2a\}$. Then for all $c \in \N$,
    $$\pr\big(d_n\big(\T_n(a), s\big)>c\big) \le \exp\left(-\frac{c a^\beta}{2n^\beta}\right).$$
\end{lem}

\begin{proof}
    First observe that we have,
    $$d_n\big(\T_n(a), s\big) = \min \big\{k \in \N: p^k(s) \in [a]\big\}.$$
    As $k$ grows, we iteratively reveal the parent of $s$ until we reach $\T_n(a)$. For a vertex $i + 1 \in [n]$ we have two options. If $R_i = 0$ then $p(i+1) = i$ and if $R_i = 1$ then $p(i+1) \in_u [i]$. In particular, assume that $p^{k-1}(s) \notin [a]$ so that we have not yet found the full ancestral line from $s$ to $\T_n(a)$. In this case,
    \begin{align*}
        i) \;\;  &\pr\big(R_{p^{k-1}(s) - 1} = 1 \; \big| \; s, p(s), \dots, p^{k-1}(s)\big) = \left(\frac{p^{k-1}(s) - 1}{n}\right)^\beta \ge \left(\frac{a}{n}\right)^\beta,\\
        ii)\; &\pr\big(p^k(s) \in [a] \; \big| \; s, p(s), \dots, p^{k-1}(s) \text{ and } R_{p^{k-1}(s) - 1} = 1 \big) \ge \frac{a}{p^{k-1}(s) - 1} \ge \frac{1}{2}.
    \end{align*}
    In words, with probability at least $\big(\frac{a}{n}\big)^\beta$ the parent $p^k(s)$ of $p^{k-1}(s)$ is chosen uniformly from $[p^{k-1}(s) - 1]$ and when this happens, it is attached to a vertex in $[a]$ with probability at least $\frac{1}{2}$. Combining $i)$ and $ii)$ yields,
    $$\pr\big(p^k(s) \in [a] \; \big| \; s, p(s), \dots, p^{k-1}(s)\big) \ge \frac{a^\beta}{2n^\beta}.$$
    Putting all results together, we find,
    \begin{align*}
        \pr\big(d\big(\T_n(a), s\big) > c\big) &= \pr\big(\!\min\{k:p^k(s) \in [a]\} > c\big),\\
        &\le \left(1 - \left(\frac{a^\beta}{2n^\beta}\right)\right)^c,\\
        &\le \exp\!\left(\!-\frac{ca^\beta}{2n^\beta}\right),
    \end{align*}
    which is the desired result.
\end{proof} 
This result can be used to obtain a bound on $\pr(d_H(\T_n(a), \T_n(2a))>c)$.
\begin{lem}
\label{Lemma: bound dubbling discrete uniform}
    Fix $a \in \{1, \dots, \floor{n/2}\}$ and $c \in \N$. We have,
    $$\pr\big(d_H\big(\T_n(a), \T_n(2a)\big)>c\big) \le \frac{2a}{c}\exp\!\left(\!-\frac{a^\beta c}{4n^\beta}\right)$$
\end{lem}
\begin{proof}
    From Lemma \ref{Lemma: distance vertex tree discrete case}, we immediately get
    $$\E\!\left[\#\left\{s \in \{a+1, \dots, 2a\} : d_n\big(\T_n(a), s\big)> \frac{c}{2}\right\} \right] \le a\exp\left(-\frac{a^\beta c}{4n^\beta}\right).$$
    If a vertex $v \in \T_n(2a) \setminus \T_n(a)$ satisfies $d\big(\T_n(a), v\big) > c$, then all $x \in \big\{v, p(v), \dots, p^{\lceil c/2\rceil -1}(v)\big\}$ satisfy $d\big(\T_n(a), x\big) > \frac{c}{2}$. Hence,
    $$\E\!\left[\#\left\{s \in \{a+1, \dots, 2a\} : d\big(\T_n(a), s\big)\right\} > \frac{c}{2}\right] \ge \frac{c}{2}\pr\big(d_H\big(\T_n(a), \T_n(2a)\big) > c\big).$$
    Combining both bounds gives the desired result.
\end{proof}

\begin{proof}[Proof of Theorem \ref{Theorem: compactness discrete tree uniform case}]
    The theorem is shown upon finding $\epsilon_i : [0,\infty) \to [0, \infty)$ such that as $t\to \infty$, both
    $$\sum_{i=0}^\infty \epsilon_i(t) \qquad \text{ and } \qquad  \limsup_{n \to \infty}\sum_{i=0}^\infty \pr\big(d_H\big(\T_n(2^itn^\frac{\beta}{\beta+1}), \T_n(2^{i+1}tn^\frac{\beta}{\beta+1})\big) > \epsilon_i n^\frac{\beta}{\beta+1}\big) $$
    tend to $0$.
    Take $\epsilon_i(t) = \frac{4(i+1)(\beta+1)}{(2^it)^\beta}t^{\frac{\beta}{2}}$ and substitute $a = 2^itn^\frac{\beta}{\beta+1}$ and $c = \epsilon_i(t)n^\frac{\beta}{\beta+1}$ in Lemma \ref{Lemma: bound dubbling discrete uniform}. After this, the computations are completely analogous to that in Theorem~\ref{Theorem: limit continuous tree is 0}.
\end{proof}

\subsection{Prokhorov approximation of $(\T,\mu)$ by its first branches -- polynomial case}
\label{Section: convergence measures continuous polynomial case}
\begin{thm}
\label{prokhorov convergence continuous measure}
    There exists a probability measure $\mu$ on $\T$ such that for all $\epsilon>0$,
    $$\lim_{k \to \infty}\pr\!\left(d_P\left(\mu^{(k)}, \mu\right)>\epsilon)\right) = 0.$$
\end{thm}
Recall that $\T$ is constructed as subset of Polish space $\ell^1$ and thus the set of probability measures on $\ell^1$ together with Prokhorov distance, denoted $(\mathcal{P}(\ell^1), d_P)$, also forms a Polish space. Hence, it suffices to show $\mu^{(k)}$ is Cauchy in probability. 

\begin{definition}
    For $A \subset \T^{(k)}$, let $A^\uparrow = \pi_k^{-1}(A) \cap \T$, where $\pi_k$ is the projection map. I.e. $\pi_k(x_1, x_2, \dots) = (x_1, \dots, x_k, 0 \dots) \in \ell^1$. Thus, $A^\uparrow$ consists of $A$ and points in $\T \setminus \T^{(k)}$ whose path to $\T^{(k)}$ ends in $A$.
\end{definition}

\begin{lem}
\label{Lemma: Martingale property}
    Let $A \subset \T^{(k)}$. Then $\mu^{(j)}(A^\uparrow)$ is a martingale for $j \ge k$ in filtration $\sigma(\T^{(j)})$.
\end{lem}

\begin{proof}
    It is clear that $\mu^{(j)}$ is $\sigma(\T^{(j)})$ measurable and $\E[|\mu^{(j)}(A^\uparrow)|] < \infty$ as $\mu^{(j)}(A^\uparrow) \in [0,1]$. It remains to show $\E[\mu^{(j+1)}(A^\uparrow) \mid \mathcal{F}_j] = \mu^{(j)}(A^\uparrow)$. Define $G_j = \sigma(\sigma(\T^{(j)}), c_{j+1}, C_{j+1})$ where $c_j = C_j - C_{j-1}$ is the length of the $j$th branch. Observe that conditional on $\T^{(j)}$ the $(j+1)$st branch is part of $A^\uparrow$ with probability $\mu^{(j)}(A^\uparrow)$. Depending on whether branch $j+1$ is added to $A^\uparrow$ or not, we see,
    $$\mu^{(j+1)}(A^\uparrow) = \begin{cases}
        (C_j\mu^{(j)} + c_{j+1})/C_{j+1} & \text{ with probability } \mu^{(j)}(A^\uparrow),\\
        C_j\mu^{(j)}/C_{j+1} & \text{ with probability } 1 - \mu^{(j)}(A^\uparrow).
    \end{cases}$$
    By putting the two cases together, we obtain,
$$ \E\!\left[\mu^{(j+1)}(A^\uparrow) \; \big| \; G_j\right] = \mu^{(j)}(A^\uparrow) \frac{C_j\mu^{(j)}(A^\uparrow) + c_{j+1}}{C_{j+1}} + (1 - \mu^{(j)}(A^\uparrow)) \frac{C_j\mu^{(j)}(A^\uparrow)}{C_{j+1}} =  \mu^{(j)}(A^\uparrow).
$$
By the tower property, we obtain,
$$\E\!\left[\mu^{(j+1)}(A^\uparrow) \; \middle| \;\sigma(\T^{(j)})\right] = \E\!\left[\E\!\left[\mu^{(j+1)}(A^\uparrow) \; \middle| \; G_j\right] \; \middle| \;\sigma(\T^{(j)})\right]  = \mu^{(j)}(A^\uparrow),$$
where we used that $\mu^{(j)}(A^\uparrow)$ is $\sigma(\T^{(j)})$ measurable. This concludes the proof.
\end{proof}
In particular, since $\mu^{(j)}(A^\uparrow)$ is a bounded martingale, we note that $\mu^{(j)}(A^\uparrow)$ converges to some limit. Before continuing, recall the following basic fact on the Prokhorov distance between two probability measures.

\begin{lem}
\label{Useful lemma prokhorov and partitions}
    Let $\mu, \nu$ be two Borel measures on the same metric space, and suppose $K_1, \dots, K_N$ is a partition of the support of $\mu$ such that diam$(K_i) \le \epsilon$ for all $i = 1, \dots, N$. We have,
    $$d_P(\mu, \nu) \le \max\left\{\epsilon, \sum_{i=1}^N |\mu(K_i) - \nu(K_i)|\right\} \le \epsilon + \sum_{i=1}^N |\mu(K_i) - \nu(K_i)|.$$
\end{lem}

\begin{thm}
\label{Theorem: convergence measure continuous uniform}
    The sequence $\big(\mu^{(j)}\big)_{j \in \N}$ is Cauchy in probability in $(\mathcal{P}(\ell^1), d_P)$.
\end{thm}
\begin{proof}
    Fix $\epsilon>0$ and let $K$ be large enough so that $\pr(d_H(\T^{(K)}, \T)>\epsilon ) < \epsilon$, which is possible by Theorem \ref{Theorem: limit continuous tree is 0}. Condition on the event $\T^{(K)}$ and let $J_1, \dots, J_{N_\epsilon}$ be a measurable partition of $\T^{(K)}$ for which $\text{diam}(J_i) < \epsilon$ for all $i$, where $N_\epsilon$ is the smallest number of measurable sets needed in the partition. We set $J_i = \emptyset$ for $i > N_\epsilon$. Conditional on $E = \big\{d_H\big(\T^{(K)}, \T\big)\le\epsilon\big\}$, we have $\text{diam}(J_i^\uparrow) \le 3\epsilon$ and thus the family $\big(J_i^\uparrow\big)_{i \in [N_\epsilon]}$ is a partition of $\T$ of sets of diameter at most $3\epsilon$. It follows from Lemma \ref{Useful lemma prokhorov and partitions} that,
$$d_P\big(\mu^{(j)}, \mu^{(m)}\big) \le \sum_{i=1}^{N_\epsilon} |\mu^{(j)}(J_i^\uparrow) - \mu^{(m)}(J_i^\uparrow)| + 3\epsilon + \mathbbm{1}_{\{E^c\}.}$$
By applying union bounds, we get for all $\delta> 0$ and $b > 0$,
\begin{align*}
&\pr\!\left(d_P(\mu^{(j)}, \mu^{(m)}) \ge \delta N_\epsilon + 3\epsilon \; \middle| \; \T^{(K)} \right),\\
\le \; &\sum_{i=1}^{b} \pr\!\left(|\mu^{(j)}(J_i^\uparrow) - \mu^{(m)}(J_i^\uparrow)|>\delta \; \middle| \; \T^{(K)}\right)\mathbbm{1}_{\{N_\epsilon \le b\}} + \mathbbm{1}_{\{N_\epsilon > b\}} + \pr(E^c \mid \T^{(K)}).
\end{align*}
We set $\delta = \frac{\epsilon}{N_\epsilon}$ and take expectations to obtain that for all $b > 0$,
\begin{align*}
    \pr(&d_P(\mu^{(j)}, \mu^{(m)}) \ge 4\epsilon ) \le \sum_{i=1}^b \pr\!\left(|\mu^{(j)}(J_i^\uparrow) - \mu^{(m)}(J_i^\uparrow)| > \frac{\epsilon}{b}\right) + \pr(N_\epsilon > b) + \epsilon.
\end{align*}
Recall that $N_\epsilon$ is the minimal number of sets $J_i$ of diameter $\epsilon$ needed to partition $\T^{(K)}$. Since $\T^{(K)} \subset \T$ is compact almost surely, we can make $\pr(N_\epsilon > b)$ arbitrarily small with finite $b$. Since $\mu^{(j)}(A^\uparrow)$ is a bounded martingale, we can make $\pr\!\left(|\mu^{(j)}(J_i^\uparrow) - \mu^{(m)}(J_i^\uparrow)| > \frac{\epsilon}{b}\right)$ arbitrarily small by taking~$n,m$ large enough. Hence $\pr(d_P(\mu^{(j)}, \mu^{(m)}) > 4\epsilon) < \epsilon$, for $j,m$ large enough which shows $\mu^{(j)}$ is Cauchy in probability.
\end{proof}

\subsection{Prokhorov approximation of $(\T_n,\nu_n)$ by its first branches -- polynomial case}
\label{Section: convergence measures discrete polynomial case}
\begin{thm}
\label{Theorem Prokhorov convergence}
    For all $\epsilon > 0$, we have
    $$\lim_{k \to \infty} \limsup_{n \to \infty} \pr(d_P(\nu_n^{(k)}, \nu_n) > \epsilon) = 0.$$
\end{thm}

\begin{definition}
    For $A \subset \T_n^{(k)}$, let $A^\uparrow$ be the set $A$ together with all vertices $v \in \T_n$ whose path from $v$ to $\T_n^{(k)}$ ends in $A$. (This in general does not agree with the set of descendants of vertices in $A$).
\end{definition}
Then, the following lemma can be proved analogous to Lemma \ref{Lemma: Martingale property}.
\begin{lem}
    Let $A \subset \T_n^{(k)}$. Then $(\nu_n^{(j)}(A^\uparrow))_{j\ge k}$ is a martingale in filtration $(\sigma(\T_n^{(j)}))_{j\ge k}$.
\end{lem}

We would like to prove Theorem \ref{Theorem Prokhorov convergence} in a similar manner to the proof of Theorem \ref{prokhorov convergence continuous measure}. In this proof, we used that $\mu^{(j)}(A^\uparrow)$ trivially converges as it is a bounded martingale. To adapt the proof to the current setting, it does not suffice that, for fixed $n$, $\nu_n^{(j)}(A^\uparrow)$ is a bounded martingale. Firstly, we need that the limit is precisely $\nu_n$, and, secondly, we need bounds uniformly in all $n$ large enough. The following lemma states precisely that.
\begin{lem}
\label{lemma: convergence measures discrete case groundwork}
Let $A \subset \T_{n}^{(k)}$. We have,
$$\pr\!\left(|\nu_n^{(k)}(A^\uparrow) - \nu_n(A^\uparrow)| \ge t \; \middle| \; \T_n^{(k)}\right) \le \frac{X_{n,k}}{t^2},$$
where $X_{n,k}$ is a random variable measurable with respect to $\T_n^{(k)}$, not dependent on the choice for $A$, such that, 
$$\limsup_{n \to \infty} \E[X_{n,k}] \le \frac{C}{k},$$
for some positive constant $C$.
\end{lem}

We first assume the above lemma holds, and use this to prove Theorem \ref{Theorem Prokhorov convergence}.

\begin{proof}[Proof of Theorem \ref{Theorem Prokhorov convergence}]
Fix $\epsilon > 0$ and $K$ large enough so $\pr(d_H(\T_n^{(K)}, \T_n)>\epsilon n^\frac{\beta}{\beta+1}) < \epsilon$ for all $n$ large enough, which is possible by Theorem \ref{Theorem: compactness discrete tree uniform case}. Write $N_\epsilon$ for the size of the smallest measurable partition $J_1, \dots, J_{N_\epsilon}$ of $n^{-\frac{\beta}{\beta+1}}\T_n^{(K)}$ where all sets have diameter at most $\epsilon$, taking $J_i=\emptyset$ for $i>N_\epsilon$. 
By using the exact same reasoning as in the proof of Theorem \ref{Theorem: convergence measure continuous uniform}, we obtain that for all $b > 0$ and $k \ge K$,
\begin{align*}
    \pr\big(d_P(\nu_n, \nu_n^{(k)}) \ge 4\epsilon \; \big| \; \T_n^{(k)}\big) &\le \sum_{j=1}^b \pr\!\left(|\nu_n^{(k)}(J_j^\uparrow) - \nu_n(J_j^\uparrow)| \ge \frac{\epsilon}{b} \; \middle| \; \T_n^{(k)}\right),\\
    &+ \mathbbm{1}_{\{{N_\epsilon} >b\}} + \pr(E_n^c \mid \T_n^{(k)}),
\end{align*}
where $E_n = \big\{d_H\big(\T_n^{(K)}, \T_n\big) \le \epsilon n^\frac{\beta}{\beta+1}\big\}$. By applying Lemma \ref{lemma: convergence measures discrete case groundwork}, followed by taking expectations and limits, we obtain,
\begin{align*}
    \limsup_{n \to \infty} \pr\!\left(d_P(\nu_n, \nu_n^{(k)}) \ge 4\epsilon \right) &\le \frac{Cb^3}{\epsilon^2 k} + \limsup_{n \to \infty}\pr(N_\epsilon > b) + \limsup_{n \to \infty}\pr(E_n^c).
\end{align*}
By Theorem \ref{Theorem: finite tree convergence} we can make $\pr(N_\epsilon > b)$ arbitrarily small for all $n$ large enough by picking $b$ large enough as $n^{-\frac{\beta}{\beta+1}}\T_n^{(K)}$ converges to $\T^{(K)}$ which is compact a.s. By assumption, $\pr(E_n^c) < \epsilon$ and the first term can be made arbitrarily small by taking $k$ large.
\end{proof}

It remains to prove Lemma \ref{lemma: convergence measures discrete case groundwork}. 

\begin{proof}[Proof of Lemma \ref{lemma: convergence measures discrete case groundwork}]
    Recall that $N$ denotes the number of branches of $\T_n$ and $C_k^n = n$ for $k 
\ge N$.  We first show that $X_{n,k} = \sum_{j=k}^{n-1} \E\!\left[\frac{2n^{2\beta}}{(C_j^n)^{2\beta + 2}} \; \middle| \; \T_n^{(k)}\right]$ satisfies
    $$\pr\!\left((\nu_n^{(k)}(A^\uparrow) - \nu_n(A^\uparrow))^2 \ge t^2 \; \middle| \; \T_n^{(k)}\right) \le \frac{X_{n,k}}{t^2}$$
    for all measurable $A$. To this end, apply Markov's inequality to obtain,
    $$\pr\!\left((\nu_n^{(k)}(A^\uparrow) - \nu_n(A^\uparrow))^2 \ge t^2 \; \middle| \; \T_n^{(k)}\right) \le \frac{1}{t^2}\E\!\left[(\nu_n^{(k)}(A^\uparrow) - \nu_n(A^\uparrow))^2 \; \middle| \; \T_n^{(k)}\right].$$
    Recall that we have $\nu_n^{(N)} = \nu_n$ and by convention $\nu_n^{(k)} = \nu_n$ for $k \ge n$. Since $N \le n$ holds deterministically, we have,
    \begin{align*}
        \E\!\left[(\nu_n^{(k)}(A^\uparrow) - \nu_n(A^\uparrow))^2 \; \middle| \; \T_n^{(k)}\right] &= \E\!\left[\nu_n^{(k)}(A^\uparrow)^2 + \nu_n^{(n)}(A^\uparrow)^2 - 2\nu_n^{(k)}(A^\uparrow)\nu_n^{(n)}(A^\uparrow)\; \middle| \; \T_n^{(k)}\right],\\
        &=\E\!\left[\nu_n^{(n)}(A^\uparrow)^2 - \nu_n^{(k)}(A^\uparrow)^2\;\middle|\; \T_n^{(k)}\right],\\
        &= \sum_{j=k}^{n-1}\E\!\left[\nu_{n}^{(j+1)}(A^\uparrow)^2 - \nu_n^{(j)}(A^\uparrow)^2 \; \middle| \; \T_n^{(k)}\right],
    \end{align*}
     Write $c_j^n = C_j^n - C_{j-1}^n$ for the number of vertices on branch $j$. We obtain,
    \begin{align*}
    &\E\!\left[\nu_n^{(j+1)}(A^\uparrow)^2 - \nu_n^{(j)}(A^\uparrow)^2 \; \middle|\; \T_n^{(k)}\right]\\
    =\;&\E\!\left[\nu_n^{(j)}(A^\uparrow) \left(\frac{C^n_j \nu_n^{(j)}(A^\uparrow) + c^n_{j+1}}{C^n_{j+1}}\right)^2 \hspace{-3pt} + (1 - \nu_n^{(j)}(A^\uparrow)) \left(\frac{C^n_j \nu_n^{(j)}(A^\uparrow)}{C^n_{j+1}}\right)^2 - \nu_n^{(j)}(A^\uparrow)^2 \; \middle| \; \T_n^{(k)} \right]\\
    =\;&\E\!\left[\frac{(c^n_{j+1})^2\nu_n^{(j)}(A^\uparrow)(1-\nu_n^{(j)}(A^\uparrow))}{(C^n_{j+1})^2} \; \middle| \; \T_n^{(k)} \right]\\
    \le\;&\E\!\left[\frac{(c_{j+1}^n)^2}{(C_{j}^n)^2} \; \middle| \; \T_n^{(k)}\right],    
\end{align*}
Observe that $c_{j+1}^n > x$ given $\T_n^{(j)}$ happens precisely when all of $R_{C_j^n + 1}, \dots, R_{C_j^n + x}$ equal zero, which happens with probability,
    $$\pr\big(c_{j+1}^n > x \; \big| \; \T^{(j)}_n\big) = \prod_{i=1}^x \left(1 - \left(\frac{C_j^n + i}{n}\right)^\beta\right) \le \left(1 - \left(\frac{C_j^n}{n}\right)^\beta\right)^x.$$
    Thus, conditional on $\T^{(j)}_n$, it holds that $c_{j+1}^n $ is dominated by $Y \sim \text{Geom} \hspace{-1pt} \Big(\left(\frac{C_j^n}{n}\right)^\beta\Big)$, which gives the bound,
    $$\E\!\left[\big(c^n_{j+1}\big)^2 \; \middle| \; \T_n^{(j)}\right] \le \E[Y^2] \le \frac{2n^{2\beta}}{(C_j^n)^{2\beta}}.$$
    By using the tower property of expectation, we obtain,
    \begin{align*}
        \pr\!\left((\nu_n^{(k)}(A^\uparrow) - \nu_n(A^\uparrow))^2 \ge t^2 \; \middle| \; \T_n^{(k)}\right)&\le\frac{1}{t^2} \sum_{j=k}^{n-1}\E\!\left[\frac{(c^n_{j+1})^2}{(C_{j}^n)^2} \; \middle| \; \T_n^{(k)}\right],\\
        &\le \frac{1}{t^2}\sum_{j=k}^{n-1} \E\!\left[\frac{2n^{2\beta}}{(C_j^n)^{2\beta + 2}} \; \middle| \; \T_n^{(k)}\right] = \frac{X_{n,k}}{t^2}.
    \end{align*}
    It only remains to show $\limsup_{n \to \infty} \E[X_{n,k}] \le \frac{C}{k}$ for some constant $C$. For this, observe that,
    \begin{equation}\label{Equation: bound C_j by binomial polynomial case}\pr(C_j^n < x ) \le \pr(X \ge j), \text{ where } X \sim \text{Binom}\!\left(x, \left(\frac{x}{n}\right)^\beta\right).\end{equation}
    Indeed, $C_j^n < x$ can only happen if at least $j$ of the Bernoulli variables $R_1, \dots, R_x$ produce a $1$. The bound follows since $\pr(R_i = 1) \le \left(\frac{x}{n}\right)^\beta$ for all $i \le x$. Via a standard Chernoff bound on the tail of the binomially distributed $X$, we obtain,
    \begin{equation}\label{Equation: Chernoff bound polynomial case}\pr\big(C_j^n \le x\big) \le \pr(X \ge j) \le \left(\frac{e\E[X] }{j}\right)^j = \left(\frac{ex^{\beta+1}}{jn^{\beta}}\right)^j.\end{equation}
    Using this, we compute,
    \begin{align*}
        \E\!\left[\frac{1}{(C_j^n)^{2\beta + 2}}\right] &\le \sum_{k=1}^\infty \left(\frac{1}{k^{2\beta+2}} - \frac{1}{(k+1)^{2\beta + 2}}\right)\pr(C_j^n \le k),\\
        &\le C\sum_{k=1}^{M} k^{-2\beta - 3} \left(\frac{ek^{\beta+1}}{jn^{\beta}}\right)^j  + \sum_{k = M+1}^\infty\left(\frac{1}{k^{2\beta+2}} - \frac{1}{(k+1)^{2\beta + 2}}\right),\\
        &\le C\left(\frac{e}{jn^\beta}\right)^j \sum_{k=1}^M k^{j(\beta+1) - 2\beta - 3} + M^{-2\beta - 2},\\
        &\le C\left(\frac{e}{jn^\beta}\right)^j \big(M^{\beta+1}\big)^{(j-2)} + \big(M^{(\beta+1)}\big)^{-2}
    \end{align*}
    where the above computation holds for $j \ge 3$ and $C$ is some positive constant. By choosing $M^{\beta+1} = n^\beta j e^{-1}$, we obtain,
    \begin{equation} \label{Equation: bound on expectation of C_j^n}\E\!\left[\frac{1}{(C_j^n)^{2\beta + 2}}\right] \le C j^{-2}n^{-2\beta},\end{equation}
    By putting everything together, we obtain,
    $$\limsup_{n \to \infty}\E\!\left[\sum_{j=k}^N \E\!\left[\frac{2n^{2\beta}}{(C_j^n)^{2\beta + 2}} \; \middle| \; \T_n^{(k)}\right]\right] \le \limsup_{n \to \infty} \sum_{j=k}^n \E\!\left[\frac{2n^{2\beta}}{(C_j^n)^{2\beta + 2}}\right] \le C\sum_{j=k}^\infty \frac{1}{j^2} \le \frac{C}{k},$$
    which finishes the proof. 
\end{proof}

We end this section with images of the tree $\T_n$ for $n = 800$ for various values of $\beta$.
\begin{figure}[h!]
    \centering
    \begin{subfigure}[b]{0.325\textwidth}
        \centering
        \includegraphics[width=\textwidth]{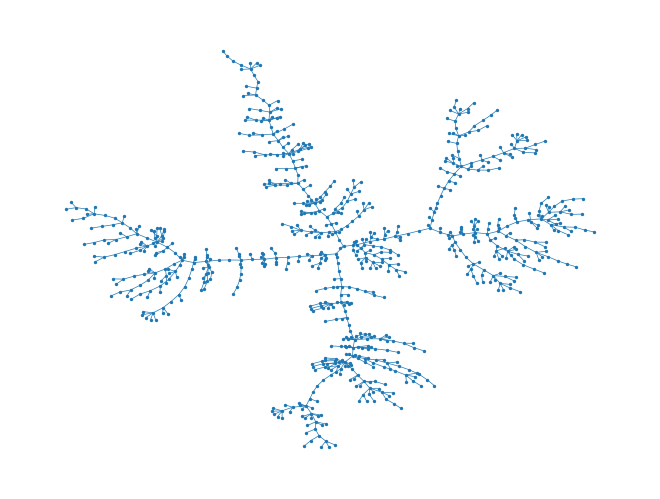}
        \caption*{$\beta = 0.8$}
        \label{fig:sub1}
    \end{subfigure}
    \begin{subfigure}[b]{0.325\textwidth}
        \centering
        \includegraphics[width=\textwidth]{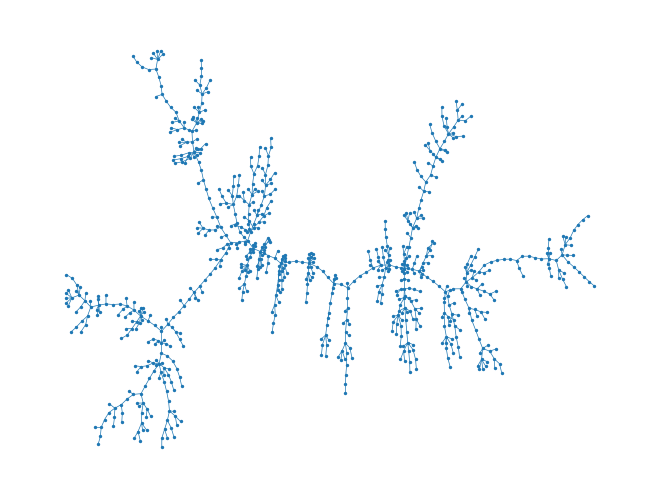}
        \caption*{$\beta = 1$}
        \label{fig:sub2}
    \end{subfigure}
    \begin{subfigure}[b]{0.325\textwidth}
        \centering
        \includegraphics[width=\textwidth]{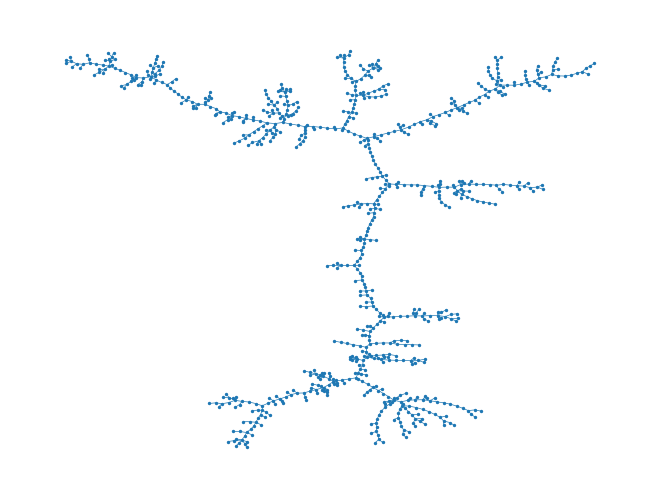}
        \caption*{$\beta = 1.5$}
        \label{fig:sub3}
    \end{subfigure}
    
    \vspace{0.5em} 
    
    \begin{subfigure}[b]{0.325\textwidth}
        \centering
        \includegraphics[width=\textwidth]{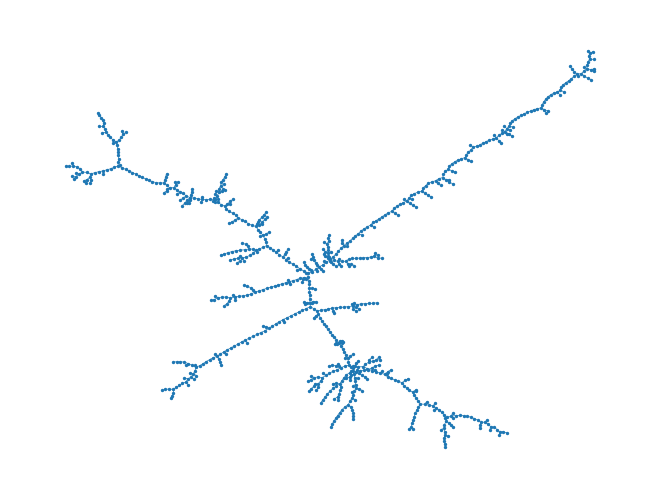}
        \caption*{$\beta = 2.5$}
        \label{fig:sub4}
    \end{subfigure}
    \begin{subfigure}[b]{0.325\textwidth}
        \centering
        \includegraphics[width=\textwidth]{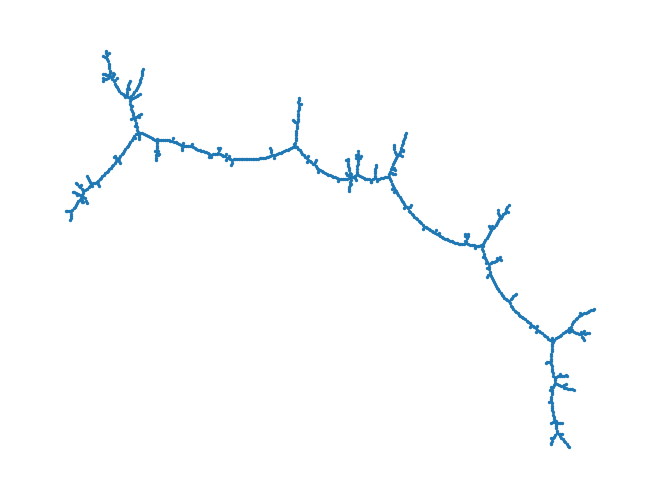}
        \caption*{$\beta = 5$}
        \label{fig:sub5}
    \end{subfigure}
    \begin{subfigure}[b]{0.325\textwidth}
        \centering
        \includegraphics[width=\textwidth]{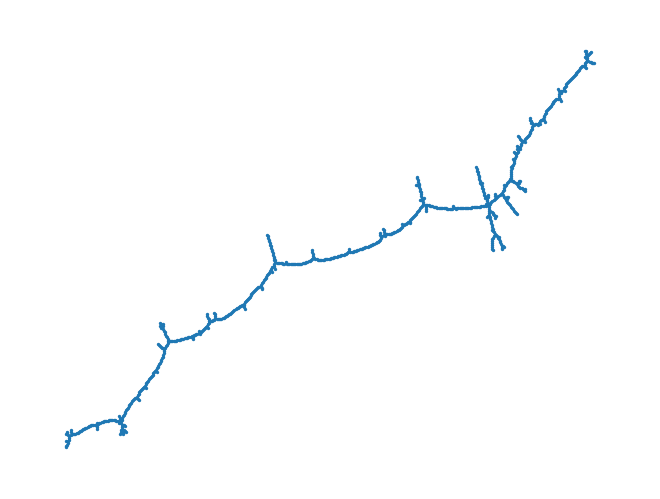}
        \caption*{$\beta = 8$}
        \label{fig:sub6}
    \end{subfigure}
    \caption{Trees $\T_n$ for various choices of $\beta$} 
    \label{fig:2x2grid}
\end{figure}

\section{The Logarithmic Case}
\label{Section: the logarithmic case}
Let $g(n,\cdot) : \{1, \dots, n-1\} \to \R_{\ge 0}$ be given by $g(n,i) = \ln^\gamma(in^{-\frac{1}{2}} + 1)n^{-\frac{1}{2}}$ and observe that $g(n,i) \in [0,1]$ for all $i \in [n-1]$ for all sufficiently large $n$. In this section, we study the discrete $g(n,i)$ aggregation tree. That is, let $R_i \sim \text{Ber}(g(n,i))$ be independent of each other, then $\T_n^g$ is constructed iteratively, by gluing vertex $i+1$ to vertex $i$ if $R_i = 0$ and gluing vertex $i+1$ to vertex with a uniform label in $[i]$ if $R_i = 1$. As before, $C_1^n < \dots < C_{N-1}^{n}$ are the indices $i$ for which $R_i = 1$, which are the cut points of the branches used to construct $\T_n$ and conditional on the cut points $B_i^n \sim \text{Unif}([C_i^n])$ independently denote the attachment points. The definitions of partial trees $\T_n(a), \T_n^{(j)}$ and measures $ \nu_n, \nu_n^{(j)}$ are as before, but with respect to $\T_n^g$ rather than $\T_n^f$. Lastly, throughout this section, we write $\T_n$ instead of $\T_n^g$ for ease of notation. 

On the continuous side, we let $\eta$ denote a PPP of intensity $\ln^\gamma(t+1) dt$ on $\R_{\ge 0}$ and order its points as $0 < C_1 < C_2 < \dots$. Conditional on $\eta$, let $B_1, B_2, \dots$ be independent with $B_i \sim \text{Unif}([0,C_i])$. Let $\T_\gamma$ be the tree obtained from the stick-breaking construction applied to points $C_i, B_i$. Lastly, let the definitions $\T^{(k)}, \mu^{(k)}$ and $\mu$ be as before, but applied to $\T_\gamma$ instead of $\T_\beta$ and simply write $\T$ for $\T_\gamma$.

\subsection{Convergence of the first branches -- logarithmic case}
The current section is dedicated to showing the following theorem.
\label{Section: convergence first branches logarithmic case}
\begin{thm}
\label{Theorem: finite dimensional convergence logarithm case}
    For all $k \in \N$ we have convergence in distribution in the GHP topology of,
    $$\big(\T_n^{(k)}, n^{-\frac{1}{2}}d_n, \nu_n^{(k)}\big) \xrightarrow[n \to \infty]{d} \big(\T^{(k)}, d, \mu^{(k)}\big).$$
\end{thm}
To prove the above result, we first show convergence of the discrete cut points and attachment points $C_i^n$ and $B_i^n$ to $C_i$ and $B_i$. The reasoning largely follows the proof of Theorem \ref{Proposition: Convergence repeat/attachment points polynomial case}, we highlight the differences. 
\begin{thm}
\label{lemma convergence repeat and attach points PPP log(t+1)beta}
    With notation as above, we have that for all $k \in \N$ and all $\gamma > 0,$
    $$n^{-\frac{1}{2}} \big(C_1^n, \dots, C_k^n, B_1^n, \dots, B_k^n\big) \xrightarrow[n \to \infty]{d} \big(C_1, \dots, C_k, B_1, \dots, B_k\big).$$
\end{thm}
We proceed via three Lemmas. The first Lemma is the corresponding version of Lemma \ref{Lemma: convergence to pdf first points PPP}.
\begin{lem}
\label{Lemma pdf of PPP log(t+1) converges}
    Fix $k \in \N$ and $0 < x_1 < \dots < x_k$. We have,
    $$n^{\frac{k}{2}}\pr\!\left(C_1^n = \floor{x_1n^\frac{1}{2}}, \dots, C_k^n = \floor{x_kn^\frac{1}{2}}\right) \xrightarrow[n \to \infty]{\text{u.c.}} f_{C_1, \dots, C_k}(x_1, \dots, x_k),$$
    where $f_{C_1, \dots, C_k}(x_1, \dots, x_k)$ is the pdf of the first $k$ points of a PPP of intensity $\ln^\gamma(t+1) dt$ and we recall that u.c. denotes that the convergence is uniform over compact sets. 
\end{lem}

\begin{proof}
    Firstly, we have $f_{C_1, \dots, C_k}(x_1, \dots, x_k) = \prod_{i=1}^k\ln^\gamma(x_i + 1)\exp\left(-\int_0^{x_k}\ln^\gamma(t+1) dt\right)$.
    Observe that $\big\{C_1^n = \floor{x_1n^\frac{1}{2}}, \dots, C_k^n = \floor{x_kn^\frac{1}{2}}\big\}$ happens precisely when the random variables $R_i \sim \text{Ber}(g(n,i))$ satisfy
    $$R_i = \begin{cases}
        1 &\text{ if } i \in \big\{\floor{x_1n^\frac{1}{2}}, \dots, \floor{x_kn^\frac{1}{2}}\big\}\\
        0 &\text{ if } i \in [\floor{x_k n^\frac{1}{2}}] \setminus \big\{\floor{x_1n^\frac{1}{2}}, \dots, \floor{x_kn^\frac{1}{2}}\big\}.
    \end{cases}$$
    Using precisely the same reasoning as in Equation \ref{Equation: computations cut points converge to pdf PPP points}, we get
    \begin{align*}
        n^\frac{k}{2}\pr\!&\left(C_1^n = \floor{x_1n^\frac{1}{2}}, \dots, C_k^n = \floor{x_kn^\frac{1}{2}}\right)\! =  n^{\frac{k}{2}}\prod_{i=1}^k \frac{g(n,\floor{x_in^\frac{1}{2}})}{1 - g(n,\floor{x_in^\frac{1}{2}})} \hspace{-4pt}\prod_{i=1}^{\floor{x_kn^\frac{1}{2}}}\hspace{-4pt}(1 - g(n,i)),\\
        &= (1 + o(1))\ln^\gamma(x_1 + 1) \cdots \ln^\gamma(x_k + 1) \prod_{i=1}^{\floor{x_kn^\frac{1}{2}}} \left(1 - \ln^\gamma(in^{-\frac{1}{2}} + 1)n^{-\frac{1}{2}}\right),
    \end{align*}
    and hence it suffices to show that,
    $$\prod_{i=1}^{\floor{x_kn^\frac{1}{2}}}\left(1 - \ln^\gamma\big(in^{-\frac{1}{2}} + 1\big)n^{-\frac{1}{2}}\right) \xrightarrow[n \to \infty]{u.c.} \exp\left(-\int_0^{x_k} \ln^\gamma(t+1)dt\right).$$
    For this, observe that,
    \begin{align*}
        \prod_{i=1}^{\floor{x_kn^\frac{1}{2}}}\left(1 - \ln^\gamma(in^{-\frac{1}{2}} + 1)n^{-\frac{1}{2}}\right)&= \exp\left(\sum_{i = 1}^{\floor{x_kn^\frac{1}{2}}} \ln\left(1 - \ln^\gamma(in^{-\frac{1}{2}} + 1)n^{-\frac{1}{2}}\right)\right),\\
        &=\exp\left(o(1)-\sum_{i=1}^{\floor{x_kn^\frac{1}{2}}} \ln^{\gamma}(in^{-\frac{1}{2}}+1)n^{-\frac{1}{2}} \right),\\
        &= \exp\left(o(1) - \int_0^{x_k} \ln^\gamma(t+1)dt\right),
    \end{align*}
    where the above convergence is uniform on compact sets. We used first-order approximation $\ln(1-x) = -x + O(x^2)$ as $x \to 0$ in the second inequality and recognized  the sum on the second line as a Riemann sum which converges uniformly to the integral as $\ln^\gamma(t+1)$ is continuous on $[0, x_k]$. This concludes the proof.  
\end{proof}
Below we state the logarithmic analogue of Lemma \ref{Lemma: attachment points are uniform}. The proof is completely analogous to the proof of Lemma \ref{Lemma: attachment points are uniform}.
\begin{lem}
For $t_1, \dots, t_k \in [0,1)$ and $0<s_1<\dots<s_k$, we have,
\begin{align*}
\pr\big(B_1^n\le t_1C^n_1,\dots, &B_k^n\le t_kC^n_k \; \big|\; C_1^n \le s_1n^{\frac{1}{2}}, \dots, C_k^n\le s_kn^{\frac{1}{2}}\Big)\\
\xrightarrow[n \to \infty]{} &\pr\big(B_1\le t_1C_1,\dots, B_k\le t_kC_k \; \big|\; C_1 \le s_1, \dots, C_k \le s_k\big).
\end{align*}
\end{lem}
Lastly, we proof Theorem \ref{lemma convergence repeat and attach points PPP log(t+1)beta}.
\begin{proof}[Proof of Theorem \ref{lemma convergence repeat and attach points PPP log(t+1)beta}.]
    It suffices to show,
    \begin{align*}
        &i) \;\; \pr\big(C_1^n \le s_1n^\frac{1}{2}, \dots, C_k^n \le s_kn^\frac{1}{2}\big) \xrightarrow[n \to \infty]{} \pr\big(C_1 \le s_1, \dots, C_k \le s_k\big),\\
        &ii) \; \pr(B_1^n \le t_1C_1^n, \dots, B_k^n \le t_kC_k^n \mid C_1^n \le s_1n^\frac{1}{2}, \dots, C_k^n \le s_kn^\frac{1}{2}),\\
        &\xrightarrow[n \to \infty]{} \pr\big(B_1 \le t_1C_1, \dots, B_k \le t_k C_k \mid C_1 \le s_1, \dots, C_k \le s_k\big).
    \end{align*} 
    Statement $i)$ follows from identical reasoning as the proof of Lemma \ref{convergence of repeat probabilities}. Statement $ii)$ follows from a similar proof to that of Lemma \ref{Lemma: attachment points are uniform}.
\end{proof}
\begin{proof}[Proof of Theorem \ref{Theorem: finite dimensional convergence logarithm case}]
This follows from identical reasoning to that in Section \ref{Section: GH convergence polynomial case}.
\end{proof}

\subsection{Non-compactness of \texorpdfstring{$\T_\gamma$}{T} for \texorpdfstring{$\gamma \le 1$}{gamma le 1}}
\label{Section: non compactness of T}
In this section, we prove Theorem~\ref{thm:not_compact}, which says that for increasing $f$ with $\int_1^\infty e^{-f(t)}dt=\infty$, the $f(t)$-aggregation tree is almost surely non-compact. This implies that $\T_\gamma$ is not compact for $\gamma \le 1$.

\begin{proof}[Proof for Theorem~\ref{thm:not_compact}]
    Let $\eta$ be a PPP on $\R_{\ge 0}$ of intensity $f(t)dt$ and let $A_n = \{\eta([n-1, n)) = 0\}$ be the event that $[n-1, n)$ is void of points. Observe that,
    $$\pr(A_n) = \exp\!\left(- \int_{n-1}^n f(t) dt\right) \ge \exp\left(-f(n)\right),$$
    as $f$ is increasing. Furthermore we have,
    $$\int_1^\infty e^{-f(t)}dt = \infty \implies \infty = \sum_{n=1}^\infty \exp(-f(n)) \le \sum_{n=1}^\infty \pr(A_n),$$
    where the implication follows from the comparison test as $e^{-f(t)}$ is decreasing. Since the events $A_n$ are independent, we may invoke the Borel--Cantelli lemma to obtain that infinitely many events $A_n$ occur almost surely. We conclude there must be a stick of infinite length, or infinitely many sticks of length exceeding $1$. 

    In the first case, the $f(t)$-aggregation tree is clearly not compact. In the second case, let $I = \{i \in \N: C_i - C_{i-1} \ge 1\}$ be the set of indices of the sticks with length exceeding 1. It follows that $|I| = \infty$. Define the sequence $\{\rho(C_i)\}_{i \in I}$ contained in the $f(t)$-aggregation tree. For any $i < j \in \N$, we see $d\big(\rho(C_i), \rho(C_j)\big)\ge1$, as the path from $\rho(C_j)$ to $\rho(C_i)$ must necessarily traverse all of $\rho([C_{j-1}, C_j])$. Hence $\{\rho(C_i)\}_{i \in I}$ cannot have a convergent subsequence and hence the $f(t)$-aggregation tree is not compact. 
\end{proof}
 We observe $\T_\gamma$ is almost surely not compact when $\gamma \le 1$ as $\int_1^\infty e^{-\ln^\gamma(t+1)}dt$ is integrable if and only if $\gamma > 1$. In what follows, we show $\T_\gamma$ is compact whenever $\gamma > 1$.
\subsection{Hausdorff approximation of $\T$ by its first branches -- logarithmic case}
\label{Section: compactness of T logarithmic setting}
In this section, we aim to show,
\begin{thm}
\label{compactness continuous tree logarithm case}
    For $\gamma > 1$ and for all $\epsilon > 0$, we have,
    $$\lim_{t \to \infty}\pr\big(d_H\big(\T(t), \T\big)>\epsilon\big) = 0.$$
\end{thm}

It turns out that naively following the results in Section \ref{Section: compactness T polynomial case} is not sufficient to prove the above result as following the exact reasoning used in Lemma \ref{Lemma: distance point to T(a)} and Lemma \ref{Lemma: bound doubling tree} yields that $\pr\big(d_H\big(\T(a), \T(2a)\big) > c\big) \le C\frac{a}{c} \exp\Big(-\frac{c\ln^\gamma(a+1)}{8}\Big)$ for some constant $C > 0$. By substituting $c = \epsilon_i(t)$ and $a = 2^it$, we obtain,
    $$\pr\big(d_H\big(\T(2^it), \T(2^{i+1}t)\big)>\epsilon_i(t)\big) \le  C\exp\left(\ln\left(\frac{2^{i+1}t}{\epsilon_i}\right)-\frac{\epsilon_i(t) \ln^\gamma(2^it+1)}{8}\right).$$
    The above bound is incompatible with the conditions in \eqref{conditions on epsilon_i}. Indeed for large $i$ and fixed $t$, we have $\ln(2^it + 1) \asymp i$ and hence, at least heuristically, the positive part in the exponent can only be overcome if asymptotically $\epsilon_i(t) i^\gamma > i$. Thus we require $\epsilon_i(t) > i^{1 - \gamma}$ asymptotically. As we require $\sum_{i=1}^\infty \epsilon_i(t) \to 0$ as $t \to\infty$, we need $\epsilon_i$ to be summable for fixed large $t$ and thus the argument cannot go through for $\gamma \le 2$.

    The solution lies in adapting the doubling argument. Observe that Theorem \ref{compactness continuous tree logarithm case} is shown by finding an increasing and diverging sequence $x_i$ with $x_0 = 1$ and $\epsilon_i : [0,\infty) \to [0, \infty)$ for which:
    \begin{equation}
        \label{Equation: conditions needed compactness continuous tree logarithmic}
            i) \lim_{t \to \infty} \sum_{i=0}^\infty \epsilon_i(t) = 0 \quad \text{ and } \quad ii) \lim_{t \to \infty} \sum_{i=0}^\infty \pr\big(d_H\big(\T(x_it), \T(x_{i+1}t)\big) > \epsilon_i(t)\big) = 0,
        \end{equation}
    In Section \ref{Section: compactness T polynomial case} we used the sequence $x_i = 2^i$. In the current section, we will use $x_i = \exp(i^\alpha)$ for some parameter $\alpha > 0$ depending on $\gamma$.

    We first state the logarithmic analogues of Lemmas \ref{Lemma: distance point to T(a)} and \ref{Lemma: bound doubling tree}. Both proofs are identical to the polynomial counterpart after replacing $\alpha^\beta$ with $\ln^\gamma(a+1)$ and $2a$ with $b$.

    \begin{lem}
    Fix $1\le a<b$ and let $l \in [a, b]$, then for all $c > 0$ we have,
    $$\pr\!\left(d\left(\T(a), \rho(l)\right)>c\right) \le 4\left(\frac{b}{a}\right)^2 \exp\left(-\frac{c\ln^\gamma(a+1)}{4}\right).$$
    \end{lem}

    \begin{lem}
\label{Lemma in showing distances tree xk xk+1 beta > 1}
    For $1 \le a < b$, we have, $\pr\!\left(d_H(\T(a), \T(b))>c\right) \le8\left(\frac{b}{a}\right)^2\frac{b}{c}\exp\left(\frac{-c\ln^\gamma(a+1)}{8}\right)$.
\end{lem}

\begin{proof}[Proof of Theorem \ref{compactness continuous tree logarithm case}.]
We aim to show statement $i)$ and $ii)$ in \eqref{Equation: conditions needed compactness continuous tree logarithmic}. By making the substitutions $a = x_it, \; b = x_{i+1}t$ and $c = \epsilon_i$ in Lemma \ref{Lemma in showing distances tree xk xk+1 beta > 1} and simplifying, condition $ii)$ can be
$$\lim_{t \to \infty} \sum_{i=0}^\infty \frac{1}{\epsilon_i} \exp\left(3\ln(x_{i+1}) + \ln(t) -\frac{\epsilon_i \ln^\gamma(x_it)}{8}\right) = 0,$$
where we replaced $\ln^\gamma(x_it+1)$ with $\ln(x_it)$ as $\ln(x_it+1) \ge \ln(x_it)$ and $\gamma > 1$. Next, we choose $\epsilon_i(t) = \frac{8}{\ln^\gamma(x_it)}\left(3\ln(x_{i+1}) + 2\ln(t) + \ln(g(i))\right)$ for ansatz function $g(i)$ with $g(i) \ge 1$ for all $i \in \N_0$. Recall that $x_i = \exp\big(i^\alpha)\big)$ so that we obtain,
\begin{align*}
    \sum_{i=0}^\infty \frac{1}{\epsilon_i} \exp\left(3\ln(x_{i+1}) + \ln(t) -\frac{\epsilon_i \ln^\gamma(x_it)}{8}\right) = \frac{1}{t}\sum_{i=0}^\infty \frac{1}{\epsilon_ig(i)},
\end{align*}
so that it suffices to show $\frac{1}{t}\sum_{i=0}^\infty \frac{1}{\epsilon_ig(i)} \to 0$ as $t \to \infty$. For this, observe that $\epsilon_i \ge \frac{1}{\ln^\gamma(x_it)}$ for $t \ge e$ and hence, 
\begin{align*}
    \frac{1}{t}\sum_{i=0}^\infty \frac{1}{\epsilon_i g(i)} \le \frac{1}{t}\sum_{i=0}^\infty \frac{\ln^\gamma(x_it)}{g(i)} &= \frac{1}{t} \left(\sum_{\{i : i^\alpha \le \ln(t)\}} \frac{(i^\alpha + \ln(t))^\gamma}{g(i)} + \sum_{\{i : i^\alpha > \ln(t)\}} \frac{(i^\alpha + \ln(t))^\gamma}{g(i)}\right),\\
    &\le \frac{1}{t}\left(2^\gamma \ln^\gamma(t) \ln^\frac{1}{\alpha}(t) + 2^\gamma \sum_{i^\alpha > \ln(t)}\frac{i^{\alpha \gamma}}{g(i)}\right)\xrightarrow[t \to \infty]{} 0,
\end{align*}
where in the last step, we chose $g(i) = (i+1)^{2 + \alpha \gamma}$, so that $\sum_{i^\alpha > \ln(t)}\frac{i^{\alpha \gamma}}{g(i)}$ is finite for all $t$. In particular, this shows that with this choice of $\epsilon_i$ we have,
$$\lim_{t \to \infty} \sum_{i=0}^\infty\pr\big(d_H\big(\T(x_it), \T(x_{i+1}t)\big)>\epsilon_i\big) = 0.$$
Lastly, we show $\sum_{i=0}^\infty \epsilon_i \to 0$ as $t \to \infty$. For this, we set $I_1 = \{i \in \N_0 : (i+1)^\alpha \le \ln(t)\}$ and $I_2 = \{i \in \N_0 : (i+1)^\alpha > \ln(t)\}$. We obtain,
\begin{align*}
    \sum_{i = 0}^\infty \epsilon_i &= 8\sum_{i=0}^\infty \frac{3(i+1)^\alpha + 2\ln(t) + \ln\big((i+1)^{\alpha\gamma+2}\big)}{\big(i^\alpha + \ln(t)\big)^\gamma},\\
    &\le 8 \sum_{I_1} \frac{5\alpha\ln(t) + (\alpha\gamma+2)\ln(\ln(t))}{\alpha\ln^\gamma(t)} + 8 \sum_{I_2} \frac{5(i+1)^\alpha + (\alpha\gamma + 2) \ln(i+1)}{i^{\alpha\gamma}},\\
    &\le C_1 \ln(t)^{1 - \gamma + \frac{1}{\alpha}} + C_2 \sum_{I_2} i^{\alpha - \alpha \gamma},
\end{align*}
where the final bound holds for $t$ large enough and positive constants $C_1$ and $C_2$ possibly depending on $\alpha$ and $\gamma$. The final expression goes to zero when $1 - \gamma + 1/\alpha< 0$ and $\alpha - \alpha\gamma < -1$. Both conditions are satisfied when $\alpha \cdot (\gamma-1)>1$. For example, we may pick $\alpha(\gamma) = \frac{2}{\gamma-1}$ which is positive whenever $\gamma > 1$. Thus we have shown that with,
$$x_i = \exp\big(i^{2/(\gamma-1)}\big) \text{ and } \epsilon_i = \frac{8}{\ln^\gamma(x_it)} \big(3\ln(x_{i+1}) + 2\ln(t) + (2 + \alpha \gamma)\ln(i+1)\big),$$ both conditions in \eqref{Equation: conditions needed compactness continuous tree logarithmic} are satisfied, finishing the proof. 
\end{proof}

\subsection{Hausdorff approximation of $\T_n$ by its first branches -- logarithmic case}
\label{Section: Compactness of T_n logarithmic setting}
In this section, we aim to show,
\begin{thm}
\label{Theorem: compactness tree discrete logarithm case}
    For all $\epsilon > 0$, we have,
    $$\lim_{t \to \infty} \limsup_{n \to \infty} \pr\big(d_H\big(\T_n(tn^\frac{1}{2}), \T_n\big)>\epsilon n^\frac{1}{2}\big) = 0.$$
\end{thm}
 As seen before, it suffices to show, that for the choices $\alpha(\gamma) = \frac{2}{\gamma-1}$, $x_i = \exp\big(i^\alpha)$ and $\epsilon_i = \frac{8}{\ln^\gamma(x_it)}\big(3\ln(x_{i+1}) + 2\ln(t) + (\alpha\gamma+2)\ln(i+1)\big)$, both
 \begin{equation}
 \label{Equation: conditions compactness discrete logarithm case}
     i) \sum_{i=0}^\infty \epsilon_i(t) \qquad \text{ and } \qquad ii)  \limsup_{n \to \infty}\sum_{i=0}^\infty \pr\big(d_H\big(\T(x_itn^\frac{1}{2}), \T(x_{i+1}tn^\frac{1}{2})\big) > \epsilon_i n^\frac{1}{2}\big)
 \end{equation}
 tend to 0 as $t \to \infty$.
\begin{lem}
\label{lemma: also to check}
    Fix integers $n^\frac{1}{2}\le a<b \le n$ and $l \in \{a, a+1, \dots, b\}$. Then for all $c > 0$ we have,
    $$\pr\!\left(d_n\left(\T_n(a), l\right)>c\right) \le 10\left(\frac{b}{a}\right)^2\exp\left(-\frac{c\ln^\gamma(an^{-\frac{1}{2}})n^{-\frac{1}{2}}}{4}\right).$$
\end{lem}

\begin{rem}The above Lemma is the logarithmic equivalent of Lemma \ref{Lemma: distance vertex tree discrete case}. In the proof of that lemma, we iteratively revealed the ancestors $p^k(l)$ of $l$ until we found a vertex in $[a]$. Specifically, we used that $p^k(l)$ has probability at least $\big(a/n\big)^\beta$ to be sampled uniformly from $[p^{k-1}(l) - 1]$ in which case $p^k(l) \in [a]$ with probability at least $1/2$. These two statements together lower bound the probability that $p^k(l) \in [a]$. Repeating this proof in the current setting does not yield a strong enough bound. Instead we reveal the ancestral line from $l$ to $[a]$ one branch at a time, and keep better track of the uniform attachment locations of the branches as opposed to merely saying that each branch has probability at least $a/b$ of being attached to $\T(a)$.\end{rem}
\begin{proof}
    Let $U_1, U_2, \dots \sim \text{Unif}([0,1])$ i.i.d.\ and set $t_1 = l$. We iteratively reveal the path from $l$ to $\T(a)$ branch by branch. To this end, define,
    $$Q(t_i) = \max\{\{C_j^n: C_j^n < t_i\} \cup \{a\}\}, \quad d_i = t_i - Q(t_i), \quad t_{i+1} = \floor{U_iQ(t_i)} + 1.
$$
    Hence $Q(t_i)$ is the cut point preceding $t_i$, or $a$, depending on which is larger. Then $d_i - 1$ is the distance from $t_i$ to the start of the branch on which $t_i$ lies as the first vertex on this branch is distributed as $Q(t_i) + 1$. Lastly, $\floor{U_i Q(t_i)} + 1 \in_u \{1, \dots, Q(t_i)\}$ so that the parent of $Q(t_i) + 1$ equals $t_{i+1}$. Thus we have that the distance between $t_{i}$ and $t_{i+1}$ is given by $d_i$. If $Q(t_i) = a$, we have reached $\T(a)$ and the remaining distance from $t_i$ to $\T_n(a)$ is $d_i = t_i - a$. Thus $d(\T_n(a), l) = \sum_{i=1}^N d_i$ where  $N = \min\{k : Q(t_k) = a\}$ is the number of branches traversed on the path from $l$ to $\T(a)$.

We first bound the tail of $N$. Observe that $t_{i+1} = \floor{U_i Q(t_i)} + 1$ and $Q(t_i) \le t_i - 1$ both hold deterministically and thus $t_{i+1} - 1 \le U_i(t_i - 1)$ holds deterministically as well. By iterating, we obtain $t_{j} - 1 \le U_1 \cdots U_{j-1} (t_1 - 1) \le U_{j-1} \cdots U_1 b$ and hence,
    $$\pr(N > j) \le \pr(t_{j} - 1 \ge a) \le \pr(U_{j-1}\dots U_1 b \ge a) \le \left(\frac{1}{1+x}\right)^{j-1}\left(\frac{b}{a}\right)^x,$$
	holds for all $x > 0$ and follows from a Chernoff bound. Next, we bound $\sum_{i=1}^j d_i$. Provided that $t_i - u > a$, the event $\{d_i > u\}$ occurs when $R_j = 0$ for all $t_i - u \le j \le t_i - 1$. Conditional on $d_1, \dots, d_{i-1}$ and the previously revealed path, the Bernoulli variables $R_i$ have not yet been revealed and thus are independent of the previous exploration. Thus,
		\begin{align*}
			\pr\!\left(t_i - Q(t_i)>u\: \middle| \: d_1, \dots, d_{i-1}, t_i\right) &= \mathbbm{1}_{\{t_i - u > a\}} \prod_{j=1}^{u} \left(1 - \ln^\gamma\left((t_i - j)n^{-\frac{1}{2}}+1\right)n^{-\frac{1}{2}}\right),\\
            &\le \left(1 - \ln^\gamma\left(an^{-\frac{1}{2}} + 1\right)n^{-\frac{1}{2}}\right)^u,\\
            &\le \left(1 - \ln^\gamma\left(an^{-\frac{1}{2}}\right)n^{-\frac{1}{2}}\right)^u,
		\end{align*}
		where we disregard the $+1$ in the logarithm as we assume $n^\frac{1}{2} < a < n$. We conclude that the $d_i$s are dominated by a Geom$(\lambda)$ distribution where $\lambda = \ln^\gamma\hspace{-3pt}\big(an^{-\frac{1}{2}}\big)n^{-\frac{1}{2}}$. Since $d_i$ can be dominated by Geometric random variables, independent of the history $d_{1}, \dots, d_{i-1}$, standard arguments give the bound,
        $$\pr\!\left(\sum_{i=1}^j d_i > c\right) \le e^{-c\theta} \left(\frac{\lambda e^{\theta}}{1 - (1-\lambda)e^\theta}\right)^j,$$
which holds for all values of $\theta$ for which $(1 - \lambda)e^\theta < 1$. To finish the argument, apply a union bound to obtain that for any $x, j$ and any such $\theta$,
$$\pr\!\left(\sum_{i=1}^N d_i > c\right) \le \pr\!\left(\sum_{i=1}^j d_i > c\right) + \pr(N>j) \le e^{-c\theta} \left(\frac{\lambda e^{\theta}}{1 - (1-\lambda)e^\theta}\right)^j + \left(\frac{1}{x+1}\right)^{j-1} \left(\frac{b}{a}\right)^x\!\!\!.$$
		The theorem is proven by choosing $x = 2, \theta = \frac{\lambda}{2}$ and $j = \floor{c\lambda/4}$ as then,
        \begin{align*}
            e^{-c\theta} \left(\frac{\lambda e^\theta}{1 - (1-\lambda)e^\theta}\right)^j + \left(\frac{1}{x+1}\right)^{j-1} \left(\frac{b}{a}\right)^x &\le e^{-\frac{c\lambda}{2}}2^{\frac{c\lambda}{4}} +  9e^{-\frac{c\lambda}{4}}\left(\frac{b}{a}\right)^2,\\
            &\le 10\left(\frac{b}{a}\right)^2e^{-\frac{c\lambda}{4}},
        \end{align*}
        where we used that $\frac{\lambda e^{\lambda/2}}{1 - (1 - \lambda)e^{\lambda/2}} \le 2$ for all $\lambda \in [0,1]$. This concludes the proof. 
\end{proof}

Using identical reasoning to that of the proof of Lemma \ref{Lemma: bound dubbling discrete uniform}, we get,
\begin{lem}
    \label{big lemma showing discrete compactness beta > 1}
		For integers $n^\frac{1}{2} \le a < b \le n$ and all $c>0$, we have,
		$$\pr\!\left(d_H\left(\T_n(a), \T_n(b)\right)>c\right) \le 20\left(\frac{b}{a}\right)^2 \frac{b}{c} \exp\left(-\frac{c\ln^\gamma\big(an^{-\frac{1}{2}}\big)n^{-\frac{1}{2}}}{8}\right).$$
	\end{lem}

\begin{proof}[Proof of Theorem \ref{Theorem: compactness tree discrete logarithm case}]
        By substituting $a = x_itn^{\frac{1}{2}}, b = x_{i+1}tn^\frac{1}{2}$ and $c = \epsilon_in^\frac{1}{2}$ into Lemma \ref{big lemma showing discrete compactness beta > 1} and simplifying, we obtain that condition $ii)$ of Equation \eqref{Equation: conditions compactness discrete logarithm case} is translated into showing,
        $$\lim_{t \to \infty}\sum_{i=0}^\infty \frac{1}{\epsilon_i} \exp\left(3\ln(x_{i+1}) + \ln(t) - \frac{\epsilon_i \ln^\gamma(x_it)}{8}\right) = 0.$$
        This condition also appeared in the proof of Theorem~\ref{compactness continuous tree logarithm case} and was shown to hold for the choices $\alpha(\gamma) = \frac{2}{\gamma-1}$, $x_i = \exp\big(i^\alpha)$, $\epsilon_i = \frac{8}{\ln^\gamma(x_it)}\big(3\ln(x_{i+1}) + 2\ln(t) + (\alpha\gamma+2)\ln(i+1)\big)$. Condition $i)$ in Equation \eqref{Equation: conditions compactness discrete logarithm case}, concluding the proof.
    \end{proof}

\subsection{Prokhorov approximation of $(\T,\mu)$ by its first branches -- logarithmic case}
\label{Section: convergence measures continuous logarithmic case}
\begin{thm}
    For $\gamma > 1$ and all $\epsilon > 0$, we have,
    $$\lim_{k \to \infty}\pr\!\left(d_P\left(\mu^{(k)}, \mu\right)>\epsilon)\right) = 0.$$
\end{thm}
\begin{proof}
    This follows from completely identical reasoning as Section \ref{Section: convergence measures continuous polynomial case}.
\end{proof}

\subsection{Prokhorov approximation of $(\T_n,\nu_n)$ by its first branches -- logarithmic case}
\label{Section: convergence measures discrete logarithmic case}
\begin{thm}
\label{Theorem: tightness measure discrete logarithm case}
    For $\gamma > 1$ and all $\epsilon > 0$, we have,
    $$\lim_{k \to \infty}\limsup_{n\to \infty}\pr\big(d_P\big(\nu_n^{(k)}, \nu_n\big)>\epsilon\big)= 0.$$
\end{thm}

The proof follows largely the same reasoning as the proof if Section \ref{Section: convergence measures discrete polynomial case}, so we highlight the differences and new bounds. Recall for $S \subset \T_{n}^{(k)}$, that $A^\uparrow \subset \T_n$ denotes the vertices from which the path to $\T_n^{(k)}$ ends in $A$. Furthermore, recall $N$ denotes the number of branches in $\T_n$ and we say $C_n^k = n$ for $k \ge N$. Following identical reasoning to the first part of the proof of Lemma \ref{lemma: convergence measures discrete case groundwork}, we obtain the following result.
\begin{lem}
\label{first lemma}
    For $A \subset \T_n^{(k)}$, we have,
    $$\pr\!\left((\nu_n^{(k)}(A^\uparrow) - \nu_n(A^\uparrow))^2 \ge t^2 \; \middle| \; \T_n^{(k)}\right) \le \frac{1}{t^2} \sum_{j=k}^{n-1} \E\!\left[\frac{2n}{(C_j^n)^{2} \ln^{2\gamma}\big(C_j^n n^{-\frac{1}{2}} + 1\big)} \; \middle| \; \T_n^{(k)}\right].$$
\end{lem}
Next, we aim to show,
\begin{lem}
\label{Lemma: to check}
    We have,
    $$\limsup_{n \to \infty}\E\!\left[\sum_{j=k}^{n-1} \E\!\left[\frac{2n}{(C_j^n)^{2} \ln^{2\gamma}\big(C_j^n n^{-\frac{1}{2}} + 1\big)} \; \middle| \; \T_n^{(k)}\right]\right] \xrightarrow[k \to \infty]{} 0.$$
\end{lem}

\begin{proof}
   We start with the bound,
    $$\pr(C_j^n \le x) \le \left(\frac{ex\ln^\gamma\big(xn^{-\frac{1}{2}} + 1\big)}{jn^{\frac{1}{2}}}\right)^j,$$
    which can be obtained through analogous reasoning used to derive Equations \eqref{Equation: bound C_j by binomial polynomial case} and \eqref{Equation: Chernoff bound polynomial case}. Define the function $h(x) = x^{-2}\ln^{-2\gamma}(x+1)$. In what follows, we derive the analogue of Equation~\eqref{Equation: bound on expectation of C_j^n}.
     \begin{align*}
         \E&\!\left[\frac{n}{(C_j^n)^{2} \ln^{2\gamma}\big(C_j^n n^{-\frac{1}{2}} + 1\big)}\right] = \E\!\left[h(n^{-\frac{1}{2}}C_j^n)\right],\\
         &= \sum_{k=1}^\infty \Big(h\big(kn^{-\frac{1}{2}}\big) - h\big((k+1)n^{-\frac{1}{2}}\big)\Big)\pr\big(h(n^{-1/2}C_j^n) \ge h(n^{-1/2}k)\big),\\
         &\le \hspace{-2pt} \sum_{k \le Mn^\frac{1}{2}} \hspace{-1pt}\left(h\big(kn^{-\frac{1}{2}}\big) - h\big((k+1)n^{-\frac{1}{2}}\big)\right)\pr(C_j^n\le  k) +\hspace{-2pt} \sum_{k >Mn^\frac{1}{2}} \hspace{-1pt}h\big(kn^{-\frac{1}{2}}\big) - h\big((k+1)n^{-\frac{1}{2}}\big),\\
         &\le \left(\frac{e}{j}\right)^j \hspace{-2pt}\sum_{k \le Mn^\frac{1}{2}} \hspace{-1pt} -h'(x_k) n^{-\frac{1}{2}} \ln^{j\gamma}\big(kn^{-\frac{1}{2}}+1\big) (kn^{-\frac{1}{2}})^j + \hspace{-2pt} \sum_{k >Mn^\frac{1}{2}} \hspace{-1pt}h\big(kn^{-\frac{1}{2}}\big) - h\big((k+1)n^{-\frac{1}{2}}\big),
     \end{align*}
     for some cutoff point $M = M(j,\gamma)$ to be determined. In the above, $h'(x_k)$ comes from the mean value theorem and $x_k \in (kn^{-\frac{1}{2}}, (k+1)n^{-\frac{1}{2}})$. For ease of notation, define,
     $$S_1 = \hspace{-2pt} \sum_{k \le Mn^\frac{1}{2}} \hspace{-1pt} -h'(x_k) n^{-\frac{1}{2}} \ln^{j\gamma}\big(kn^{-\frac{1}{2}}+1\big) (kn^{-\frac{1}{2}})^j \text{ and } S_2 = \hspace{-2pt} \sum_{k >Mn^\frac{1}{2}} \hspace{-1pt}h\big(kn^{-\frac{1}{2}}\big) - h\big((k+1)n^{-\frac{1}{2}}\big),$$
     We start with $S_1$. Since $-h'(x) = \ln^{j\gamma}(x+1)^{j\gamma}x^j$ is increasing and $x_k \le (k+1)n^{-\frac{1}{2}}$, we get,
     \begin{align*}S_1 &\le \sum_{k \le Mn^\frac{1}{2}} -h'((k+1)n^{-\frac{1}{2}})n^{-\frac{1}{2}} \ln((k+1)n^{-\frac{1}{2}}+1)^{j\gamma}((k+1)n^{-\frac{1}{2}})^j,\\
     &\le \sum_{k \le Mn^\frac{1}{2}}\int_{(k+1)n^{-\frac{1}{2}}}^{(k+2)n^{-\frac{1}{2}}} -h'(x) \ln^{j\gamma}(x+1)x^jdx,\\
     &\le \int_0^{M + 2n^{-\frac{1}{2}}} -h'(x)\ln^{j\gamma}(x+1)x^j dx,\\
     &= \frac{2}{j-2} (M + 2n^{-\frac{1}{2}})^{j-2} \ln^{(j-2)\gamma}(M + 2n^{-\frac{1}{2}}  + 1),
     \end{align*}
     where the last step holds for $j \ge 3$. We choose $M = \frac{j}{2e\ln^\gamma(j)}$. Since $j \le n$ we have $n^{-\frac{1}{2}} \le j^{-\frac{1}{2}}$ and thus we have that $M + 2n^{-\frac{1}{2}} \le \frac{j}{2e\ln^\gamma(j)} + 2j^{-\frac{1}{2}} \le \frac{j}{e\ln^\gamma(j)} \le j$ for sufficiently large $j$. We obtain,
     $$\left(\frac{e}{j}\right)^j S_1 \le \left(\frac{e}{j}\right)^j \frac{2}{j-2} \left(\frac{j}{e \ln^\gamma(j)}\right)^{j-2} \ln^{(j-2)\gamma}(j) = \frac{e^2}{j^2} \frac{2}{j-2} = O(j^{-3}).$$
     It remains to bound $S_2$, which is a telescoping sum and hence,
    $$S_2 \le h(M) = \frac{4e^2 \ln^{2\gamma}(j)}{j^2 \ln^{2\gamma}\left(\frac{j}{2e\ln^\gamma(j)} + 1 \right)} = O(j^{-2}).$$
    We conclude that $\E\!\left[h(n^{-\frac{1}{2}}C_j^n)\right] = O(j^{-2})$ and hence,
\begin{align*}
    \limsup_{n \to \infty} \,&\E\hspace{-2pt}\left[\sum_{j=k}^{n-1} \E\hspace{-2pt}\left[\frac{2n}{(C_j^n)^{2} \ln^{2\gamma}\big(C_j^n n^{-\frac{1}{2}} + 1\big)} \; \middle| \; \T_n^{(k)}\right]\right],\\
    &\hspace{35pt}\le \limsup_{n \to \infty}\sum_{j=k}^n \E\hspace{-2pt}\left[\frac{2n}{(C_j^n)^{2} \ln^{2\gamma}\big(C_j^n n^{-\frac{1}{2}} + 1\big)} \right]\hspace{-3pt} \le C\sum_{j=k}^\infty \frac{1}{j^2} \xrightarrow[k \to \infty]{} 0,
\end{align*}
as desired.
\end{proof}
\begin{proof}[Proof of Theorem \ref{Theorem: tightness measure discrete logarithm case}]
    We bounded $\pr\!\left((\nu_n^{(k)}(A^\uparrow) - \nu_n(A^\uparrow))^2 \ge C^2 \; \middle| \; \T_n^{(k)}\right) \le \frac{X_{n,k}}{C^2}$, where,
    $$X_{n,k} := \sum_{j = k}^N\E\!\left[\frac{2n}{(C_j^n)^{2} \ln^{2\gamma}\big(C_j^n n^{-\frac{1}{2}} + 1\big)} \; \middle| \; \T_n^{(k)}\right],$$
    satisfies $\E[X_{n,k}] \to 0$ as $k \to \infty$. Thus, we have exactly met all requirements used in the proof of Theorem \ref{Theorem Prokhorov convergence} and analogous reasoning finishes the proof. 
\end{proof}

\bibliographystyle{plainnat}
\bibliography{aggregation}
\appendix

\end{document}